\documentclass[a4paper,11pt, reqno]{amsart}
\usepackage{amssymb,amsthm,amsmath}
 \usepackage{enumerate}
\usepackage{ifthen}
\usepackage{graphicx}
\usepackage{cite}
\usepackage{amsmath,amscd,amssymb}
\usepackage{latexsym}
\usepackage[colorlinks,citecolor=blue,pagebackref,hypertexnames=false]{hyperref}

\numberwithin{equation}{section}
\allowdisplaybreaks[2]
\theoremstyle{plain}
\newtheorem{theorem}{Theorem}[section]

\newtheorem{lemma}[theorem]{Lemma}
\newtheorem{corollary}[theorem]{Corollary}
\newtheorem{proposition}[theorem]{Proposition}

\theoremstyle{definition}
\newtheorem{definition}[theorem]{Definition}

\theoremstyle{remark}
\newtheorem{remark}[theorem]{Remark}

\newtheorem{case[theorem]}{Case}

\def\supp{\hbox{supp\,}}
\def\norm#1.#2.{\lVert#1\rVert_{#2}}

\title[Decay estimates for a class of dispersive equations]{Decay estimates for a class of dispersive equations with partial inverse-square potentials
}
\author{Jiabin Qian}
\author{Manli Song}
\thanks{Corresponding author: Manli Song}

\address{\endgraf School of Mathematics and Statistics, Northwestern Polytechnical University, Xi'an, Shaanxi 710129, China}\email{qjb330424@mail.nwpu.edu.cn}

\address{\endgraf School of Mathematics and Statistics, Northwestern Polytechnical University, Xi'an, Shaanxi 710129, China} \email{mlsong@nwpu.edu.cn}


\date{\today}
\begin{document}
	
	\maketitle

	\allowdisplaybreaks

\begin{abstract}
Let $\mathcal{L}_a=-\Delta_x-\Delta_y+\frac{a}{2}|x|^{-2}$ with $a>0$ denote the Schr\"odinger operator on $L^2(\mathbb{R}^2_x\times \mathbb{R}^n_y)$, which involves a singular partial inverse-square potential. The purpose of this manuscript is twofold. 	First, relying on the explicit representation for the spectral measure associated with
the operator $\mathcal{L}_a$ established by Zhang-Zhang [J. Geom. Anal. \textbf{35}(3), Paper No. 71, 27pp (2025)], we investigate the decay estimate for a class of dispersive semigroups of the form $e^{it\phi(\sqrt{\mathcal{L}_a})}$, where $\phi: \mathbb{R}^+ \to \mathbb{R}$ is a smooth function. To handle the technical difficulty arising from the inhomogeneity of the phase function $\phi$, we adopt the frequency localization and the stationary phase method. In the second part of the paper, we first derive boundary Strichartz estimates for the fractional Schr\"odinger operator $e^{it\mathcal{L}_a^\nu}$, $0<\nu\neq\frac{1}{2}$. As applications of the established decay estimates, we further obtain Strichartz estimates for some concrete wave equations associated with the operator $\mathcal{L}_a$, which corresponds to  $\phi(r)=r, r^2, r^2+r^4, \sqrt{1+r^2}, \sqrt{1+r^4}$, and $r^\mu,0<\mu\neq 1$. Most notably, our results unify and simplify existing dispersive estimates for the operator $\mathcal{L}_a$, while extending the relevant theory to more general scenarios.
\end{abstract}
	\tableofcontents

	\section{Introduction}
    \definecolor{deepgreen}{RGB}{0,95,45}
We consider the self-adjoint Schr\"odinger operator with a partial inverse-square potential on $L^2(\mathbb{R}^{2+n})$
\begin{equation*}
\mathcal{L}_a=-\Delta_x-\Delta_y+\frac{a}{2}|x|^{-2}, \;(x,y)\in\mathbb{R}^2\times\mathbb{R}^n,\; a>0.
\end{equation*}
This operator arises naturally in the study of many-body quantum systems and distinguishes itself from the classical inverse-square potential model. It possesses a rich mathematical structure: the potential exhibits stronger singularity localized exclusively in the $x$-variables, which breaks the full rotational symmetry of the system and invalidates a wide range of symmetry-based analytical methods conventionally adopted in related studies. Beyond its mathematical features, the operator $\mathcal{L}_a$ also carries profound physical implications via its intimate connection to singular two-body quantum systems. Specifically, under the coordinate transformations $x=\frac{\sqrt{2}}{2}(x_1-x_2)$ and $y=\frac{\sqrt{2}}{2}(x_1+x_2)$, the two-particle Hamiltonian $H_{V,2}$
\begin{equation*}
H_{V,2}=-\Delta_{x_1}-\Delta_{x_2}+a|x_1-x_2|^{-2},
\end{equation*}
can be exactly transformed into the operator $\mathcal{L}_a$ defined on $\mathbb{R}^4$, which corresponds precisely to the case $n=2$. This correspondence provides a crucial foundation for investigating wave propagation dynamics in many-body quantum systems with singular interaction potentials. For the operator $\mathcal{L}_a$, the Strichartz estimates for both Schr\"odinger and wave equations have been rigorously established in the prior work of Zhang-Zhang \cite{Zhang-Zhang}.

The main aim of this paper is to investigate decay estimates for a class of dispersive equations of the form	\begin{equation}\label{DispersiveEqua}
	\begin{cases}	i\partial_tu+\phi(\sqrt{\mathcal{L}_a}) u=0,\; (t,x,y)\in\mathbb{R}\times\mathbb{R}^{2+n},\\
	u(0,x,y)=f(x,y),
	\end{cases}
	\end{equation}
where $f:\mathbb{R}^{2+n} \to \mathbb C$, and $\phi: \mathbb{R}^+ \to \mathbb{R}$ is a smooth function and we denote by (C1)-(C4) the following assumptions on $\phi$:
\\

(C1)~There exists $m_1>0$ such that for any $\alpha \geqslant 2$ and $\alpha \in \mathbb{N}$,
\begin{equation*}
|\phi'(\lambda)| \sim \lambda^{m_1-1}  , \quad |\phi^{(\alpha)}(\lambda)| \lesssim \lambda^{m_1-\alpha},\quad \lambda \geqslant 1.
\end{equation*}

(C2)~There exists $m_2>0$ such that for any $\alpha \geqslant 2$ and $\alpha \in \mathbb{N}$,
\begin{equation*}
|\phi'(\lambda)| \sim \lambda^{m_2-1}  , \quad |\phi^{(\alpha)}(\lambda)| \lesssim \lambda^{m_2-\alpha},\quad 0<\lambda<1.
\end{equation*}

(C3)~There exists $\alpha_1>0$ such that
\begin{equation*}
|\phi''(\lambda)| \sim \lambda^{\alpha_1-2}, \quad \lambda \geqslant1.
\end{equation*}

(C4)~There exists $\alpha_2>0$ such that
\begin{equation*}
|\phi''(\lambda)| \sim \lambda^{\alpha_2-2}, \quad  0<\lambda<1.
\end{equation*}
Here we use a functional calculus to define $\phi(\sqrt{\mathcal{L}_a}) $ in the following form $$\phi(\sqrt{\mathcal{L}_a})f=\mathcal{F}^{-1}\left(\phi(\rho^2+|\eta|^2) \mathcal{F} f\right), $$
where $\mathcal{F}$ denotes the distorted Fourier transform associated with the operator $\mathcal{L}_a$, defined by \eqref{Distorted} in Section \ref{sec2}.  Notice that conditions (C1) and (C3) represent the homogeneous order of $\phi$ in high frequency, and conditions (C2) and (C4) reflect the homogeneous order of $\phi$ in low frequency. If $\phi$ satisfies conditions (C1) and (C3), then $\alpha_1\leq m_1$. Similarly, if $\phi$ satisfies (C2) and (C4), then $\alpha_2\geq m_2$.

Here, we want to emphasize that several dispersive wave equations associated with the operator $\mathcal{L}_a$ reduced to the type (\ref{DispersiveEqua}). 
For instance, the Schr\"{o}dinger equation corresponds to $\phi(r)=r^2$, the wave equation corresponds to $\phi(r)=r$, the fractional Schr\"{o}dinger equation corresponds to $\phi(r)=r^\mu$, $0<\mu\neq 1$, the fourth-order Schr\"{o}dinger equation corresponds to $\phi(r)=r^2+r^4$, the beam equation $\phi(r)=\sqrt{1+r^4}$, and the Klein-Gordon equation corresponds to $\phi(r)=\sqrt{1+r^2}$, among others.  
 
Dispersive inequalities for evolution equations, such as the Schr\"odinger and wave equations, play a significant role in the study of semilinear and quasilinear problems, which naturally appear in several physical applications. Proving dispersion means estimating the decay in time on the evolution group associated with the free equation.  In recent times, dispersive properties of evolution equations have become a crucial tool in the study of a variety of questions, including local and global existence for nonlinear equations, well-posedness of Cauchy problems
for nonlinear equations in Sobolev spaces of low order, scattering theory, and many others. In many circumstances, the essential step in proving this time decay is based on the application of a stationary phase theorem on an (approximate) representation of the solution.  Dispersion phenomena, when combined with an abstract functional analysis argument, famously known as the $TT^*$-argument (see \cite{KT}), offer a range of estimates involving space-time Lebesgue norms. These inequalities, also known as Strichartz estimates, have evolved into a fundamental and amazing tool in the study of nonlinear partial differential equations over the last few decades; see \cite{Bahouri-app} and references therein.

 Strichartz estimates can also be interpreted as Fourier restriction estimates, which play a vital role in classical harmonic analysis and are closely linked to arithmetic combinatorics. Considerable attention has been devoted by several researchers to derive Strichartz estimates in different frameworks. For example, in the Euclidean setting,  these estimates have been proved for several dispersive equations, such as the wave equation and Schr\"{o}dinger equation, for instance, see the pioneering works \cite{GV, KT, Str}. The theory of Strichartz estimates has also been extended in the non-Euclidean frameworks. 
For example, see \cite{Song-Tan, BGX2000, H2005, FMV1, FV,  LS2014, SZ, Song2026, Song2016,  BKG, BBG2021, SY2023} for Strichartz estimates on H-type groups and the more generality of stratified step $2$  Lie groups, \cite{BDDM2019, DPR2010, Feng-Song, FSW, NR2005, R2008, S2013, Mejjaoli2008-1, Mejjaoli2009, Mejjaoli2013, Ben said, Ratna3, MS, FMSW} for metric measure spaces, 
\cite{AP2009, AP2014, APV2012} for hyperbolic spaces, \cite{B1993, BGT2004} for compact Riemannian manifolds,  and \cite{ILP2014} for bounded domains.

  In 1977, Strichartz \cite{Str}  derived the priori estimates for  solutions to classical Schr\"odinger, wave and Klein-Gordon equations in the space-time norm by employing the famous Fourier restriction theorem of Stein and Tomas \cite{Stein1984, T1, T2}. However, later, it was improved by Ginibre-Velo \cite{GV} using a standard duality argument along with the dispersive estimate   of the form 
 \begin{align}\label{dispersive estimate}\left\|e^{i t \phi(\sqrt{-\Delta})} u_0\right\|_X \lesssim|t|^{-\theta}\left\|u_0\right\|_{X^*},\end{align}
where $X^*$ is the dual space of $X$. In 2008, Guo-Peng-Wang \cite{GPW2008} used a unified way to study the decay for a class of dispersive semigroups $e^{it\phi(\sqrt{-\Delta})}$ on $\mathbb{R}^n$.  They assumed that $\phi: \mathbb{R}^+ \to \mathbb{R}$ is smooth satisfying (C1)-(C4) and obtained several decay estimates in time for the dispersive semigroup $e^{it\phi(\sqrt{-\Delta})}$ by introducing the Littlewood-Paley projector operator; see  \cite[Theorem 1]{GPW2008}.

Recently, Zhang-Zhang \cite{Zhang-Zhang} investigated the Strichartz estimates for solutions of Schr\"odinger and wave equations with partial inverse-square potentials.
\begin{theorem}[\cite{Zhang-Zhang}]\label{Schrodinger-Strichartz} Fix $n\geq1$. Let $u$ be a solution of the Schr\"odinger equation 
\begin{equation}\label{Schrodinger}
	\begin{cases}	i\partial_tu+\mathcal{L}_a u=0,\; (t,x,y)\in\mathbb{R}\times\mathbb{R}^{2+n},\\
	u(0,x,y)=f(x,y).
	\end{cases}
	\end{equation}
Then there exists a constant $C>0$ such that
\begin{equation}\label{Str-schr}
\|u\|_{L_t^q(\mathbb{R},L_{x,y}^r(\mathbb{R}^{2+n}))}=\|e^{it\mathcal{L}_a}f\|_{L_t^q(\mathbb{R},L_{x,y}^r(\mathbb{R}^{2+n}))}\leq C\|f\|_{L_{x,y}^2(\mathbb{R}^{2+n})},
\end{equation}
where 
\begin{equation}\label{admissible}
2\leq q\leq \infty, 2\leq r<\infty \text{ and } \frac{2}{q}=(2+n)\left(\frac{1}{2}-\frac{1}{r}\right).
\end{equation}
\end{theorem}

\begin{theorem}[\cite{Zhang-Zhang}] \label{wave-Strichartz} Fix $n\geq1$. Let $u$ be a solution of the wave equation 
\begin{equation}\label{Schrodinger}
	\begin{cases}	\partial^2_tu+\mathcal{L}_a u=0,\; (t,x,y)\in\mathbb{R}\times\mathbb{R}^{2+n},\\
	u(0,x,y)=f(x,y),\;\partial_tu(0,x,y)=g(x,y).
	\end{cases}
	\end{equation}
Then there exists a constant $C>0$ such that
\begin{equation*}
\|u\|_{L_t^q(\mathbb{R},L_{x,y}^r(\mathbb{R}^{2+n}))}\leq C\left(\|f\|_{\dot{H}_a^s(\mathbb{R}^{2+n})}+\|g\|_{\dot{H}_a^{s-1}(\mathbb{R}^{2+n})}\right),
\end{equation*}
where $s\geq0$, $2\leq q\leq \infty$, $2\leq r<\infty$, $\frac{2}{q}\leq (1+n)\left(\frac{1}{2}-\frac{1}{r}\right)$, $s=(2+n)\left(\frac{1}{2}-\frac{1}{r}\right)-\frac{1}{q}$ and $\dot{H}_a^s(\mathbb{R}^{2+n})=\mathcal{L}_a^{-\frac{s}{2}}L^2(\mathbb{R}^{2+n})$ 
 is the homogeneous Sobolev space associated with the operator $\mathcal{L}_a$.
\end{theorem}

Motivated by the work of Zhang-Zhang \cite{Zhang-Zhang}, this paper aims to establish decay estimates for a class of dispersive equations \eqref{DispersiveEqua} associated with the operator $\mathcal{L}_a$. For a homogeneous phase function $\phi$ of order $m$, that is, $\phi(\lambda r)=$ $\lambda^m \phi(r)$ for all $ \lambda>0$, dispersive estimates can be derived straightforwardly by using a theorem of Littman and dyadic decomposition. Nevertheless, the scenario becomes substantially intricate for an inhomogeneous function $\phi$, since the scaling parameters cannot be efficiently decoupled from the time variable. Since the phases we consider are not necessarily homogeneous, we adopt the Littlewood–Paley decomposition (see \cite{GPW2008}) to overcome this technical obstacle. This approach enables us to precisely distinguish low and high-frequency behaviors across in different scales.

Let $\psi\in C_c^\infty(\mathbb{R})$ such that $0\leq \psi\leq 1$, $\supp \psi\in [1/2,1]$, and 
\begin{equation*}
\sum_{k\in \mathbb{Z}} \psi(2^{-k}\lambda)=1,\quad \varphi(\lambda):=\sum_{k\leq 0} \psi(2^{-k}\lambda).
\end{equation*}
\begin{definition}[Besov spaces associated with $\mathcal{L}_a$]
For $s\in\mathbb{R}$ and $1\leq p,r\leq\infty$, $\dot{B}^{s}_{p,r,a}(\mathbb{R}^{2+n})$ is the homogeneous Besov space associated with $\mathcal{L}_a$, defined as the set of tempered distributions $f \in \mathcal{S}'\left(\mathbb{R}^{2+n}\right)$ such that $\|f\|_{\dot{B}^{s}_{p,r,a}}<\infty$ with
\begin{equation*}
\|f\|_{\dot{B}^{s}_{p,r,a}(\mathbb{R}^{2+n})}=\begin{cases}\left(\sum\limits_{k\in\mathbb{Z}}^\infty2^{ksr}\|\psi(2^{-k}\sqrt{\mathcal{L}_a})f\|_{L^p(\mathbb{R}^{2+n})}^r\right)^{\frac{1}r},\quad 1\leq r<\infty\\
 \sup\limits_{k\in\mathbb{Z}} 2^{ks}\|\psi(2^{-k}\sqrt{\mathcal{L}_a})f\|_{L^p(\mathbb{R}^{2+n})}\,\;\quad\quad,\quad r=\infty.
 \end{cases}
	\end{equation*} 
and $B^{s}_{p,r,a}(\mathbb{R}^{2+n})$ is the inhomogeneous Besov space associated with $\mathcal{L}_a$, defined as the set of tempered distributions $f \in \mathcal{S}'\left(\mathbb{R}^{2+n}\right)$ such that $\|f\|_{B^{s}_{p,r,a}}<\infty$ with
\begin{equation*}
\|f\|_{B^{s}_{p,r,a}(\mathbb{R}^{2+n})}=\|\varphi(\sqrt{\mathcal{L}_a})f\|_{L^p(\mathbb{R}^{2+n})}+\begin{cases}\left(\sum\limits_{k=1}^\infty2^{ksr}\|\psi(2^{-k}\sqrt{\mathcal{L}_a})f\|_{L^p(\mathbb{R}^{2+n})}^r\right)^{\frac{1}r},\quad 1\leq r<\infty\\
 \sup\limits_{k\geq1} 2^{ks}\|\psi(2^{-k}\sqrt{\mathcal{L}_a})f\|_{L^p(\mathbb{R}^{2+n})}\,\;\quad\quad,\quad r=\infty.
 \end{cases}
	\end{equation*}
\end{definition}
\begin{definition}[Sobolev spaces associated with $\mathcal{L}_a$]For $s \in\mathbb{R}$, the homogeneous and inhomogeneous Sobolev space associated with $\mathcal{L}_a$ are denoted respectively by $\dot{H}_a^s(\mathbb{R}^{2+n})=\mathcal{L}_a^{-\frac{s}{2}} L^2(\mathbb{R}^{2+n})$ and $H_a^s(\mathbb{R}^{2+n})=\mathcal{L}_a^{-\frac{s}{2}} L^2(\mathbb{R}^{2+n})\cap  L^2(\mathbb{R}^{2+n})$, with the norm
 \begin{align*}
	\|f\|_{\dot{H}_a^s(\mathbb{R}^{2+n})}&=\|\mathcal{L}_a^{\frac{s}{2}}f\|_{L^2(\mathbb{R}^{2+n})}=\|f\|_{\dot{B}^{s}_{2,2,a}(\mathbb{R}^{2+n})},\\
    \|f\|_{H_a^s(\mathbb{R}^{2+n})}&=\|(1+\mathcal{L}_a)^{\frac{s}{2}}f\|_{L^2(\mathbb{R}^{2+n})}=\|f\|_{B^{s}_{2,2,a}(\mathbb{R}^{2+n})}.
	\end{align*}
\end{definition}

Then the main results of this paper are as follows:
\begin{theorem}\label{ResultTime}
	Assume $\phi:\mathbb{R}^+ \to \mathbb{R}$ is smooth. Let  $U_k(t)=e^{it\phi(\sqrt{\mathcal{L}_a})}\psi(2^{-k}\sqrt{\mathcal{L}_a}), \forall k\in\mathbb{Z}$ and $U^{low}(t)=e^{it\phi(\sqrt{\mathcal{L}_a})}\varphi(\sqrt{\mathcal{L}_a})=\sum\limits_{k\leq0} U_k(t)$.   Then the subsequent outcomes are valid.
 \begin{enumerate}
 \item[\textbf{(A)}]   For $k \geq 1$, assume that $\phi$ satisfies (C1), then
	\begin{equation*}\label{res3-2}
	\|U_k(t) f\|_{L^\infty(\mathbb{R}^{2+n})} \lesssim |t|^{-\theta}2^{k\left(n+2-m_1\theta\right)}\|f\|_{L^1(\mathbb{R}^{2+n})},\,0 \leq \theta \leq \frac{n+1}{2}.
	\end{equation*}
 In addition, if $\phi$ satisfies (C3), then
	\begin{equation*}\label{res3-3}
	\|U_k(t) f\|_{L^\infty(\mathbb{R}^{2+n})} \lesssim |t|^{-\frac{n+1+\theta}{2}}2^{k\left(n+2-\frac{m_1(n+1+\theta)}{2}-\frac{\theta(\alpha_1-m_1)}{2}\right)}\|f\|_{L^1(\mathbb{R}^{2+n})},\,   0 \leq \theta \leq 1.
	\end{equation*}
 \item[\textbf{(B)}] For $k\leq 0$, suppose $\phi$ satisfies (C2), then
	\begin{equation*}
	\|U_k(t) f\|_{L^\infty(\mathbb{R}^{2+n})} \lesssim |t|^{-\theta}2^{k\left(n+2-m_2\theta\right)}\|f\|_{L^1(\mathbb{R}^{2+n})},\,0 \leq \theta \leq \frac{n+1}{2}.
	\end{equation*}
In addition, if $\phi$ satisfies (C4), then
	\begin{equation*}
	\|U_k(t) f\|_{L^\infty(\mathbb{R}^{2+n})} \lesssim |t|^{-\frac{n+1+\theta}{2}}2^{k\left(n+2-\frac{m_2(n+1+\theta)}{2}-\frac{\theta(\alpha_2-m_2)}{2}\right)}\|f\|_{L^1(\mathbb{R}^{2+n})},\,   0 \leq \theta \leq 1.
	\end{equation*}
\item[\textbf{(C)}] If $\phi$ satisfies (C2), then
	\begin{equation}\label{res-sum1}
 \|U^{low}(t) f\|_{L^\infty(\mathbb{R}^{2+n})} \lesssim (1+|t|)^{-\theta}\|f\|_{L^1(\mathbb{R}^{2+n})},\,\theta=\min\left(\frac{n+2}{m_2},\frac{n+1}{2}\right).
	\end{equation}
In addition, if $\phi$ satisfies (C4) and $\alpha_2=m_2$, then
	\begin{equation}\label{res-sum2}
	\|U^{low}(t) f\|_{L^\infty(\mathbb{R}^{2+n})}  \lesssim (1+|t|)^{-\theta}\|f\|_{L^1(\mathbb{R}^{2+n})},\,\theta=\min\left(\frac{n+2}{m_2},\frac{n+2}{2}\right).
	\end{equation}
	 \end{enumerate}
\end{theorem}
 We shall prove the above result in Section \ref{Proof of Main}.  In this paper, we use the Littlewood-Paley decomposition associated with the operator $\mathcal{L}_a$ to better understand the regularity required at low and high frequencies to obtain a time decay. From the spectral measure associated with the operator $\mathcal{L}_a$, the decay estimates mentioned above are reduced to a list of oscillatory integrals. In order to estimate these integrals, we rely on the van der Corput lemma. 

From \eqref{Str-schr}, we deduce the following frequency-localized strichartz estimate
\begin{equation}\label{Str-local}
\|e^{it\mathcal{L}_a}\psi(\sqrt{\mathcal{L}_a})f\|_{L_t^q(\mathbb{R},L_{x,y}^r(\mathbb{R}^{2+n}))}\leq C\|f\|_{L_{x,y}^2(\mathbb{R}^{2+n})},
\end{equation}
where $(q,r)$ satisfies the condition \eqref{admissible}. By virtue of the Littlewood-Paley theory, we can get from \eqref{Str-local} the following frequency-global Strichartz estimates 
\begin{equation}\label{Str-H}
\|e^{it\mathcal{L}_a}f\|_{L_t^q(\mathbb{R},L_{x,y}^r(\mathbb{R}^{2+n}))}\leq C\|f\|_{\dot{H}_a^s(\mathbb{R}^{2+n})},
\end{equation}
where $(q,r)$ and the Sobolev regularity index $s$ satisfy
\begin{equation*}
2\leq q\leq \infty, 2\leq r<\infty, \frac{2}{q}\leq (2+n)\left(\frac{1}{2}-\frac{1}{r}\right), \text{ and } s=(2+n)\left(\frac{1}{2}-\frac{1}{r}\right)-\frac{2}{q}.
\end{equation*}
However, the boundary case $r=\infty$ in \eqref{Str-H} requires separate consideration. More precisely, the boundary Strichartz estimate 
\begin{equation*}
\|e^{it\mathcal{L}_a}f\|_{L_t^q(\mathbb{R},L_{x,y}^\infty(\mathbb{R}^{2+n}))}\leq C\|f\|_{\dot{H}_a^{\frac{2+n}{2}-\frac{2}{q}}(\mathbb{R}^{2+n})},
\end{equation*}
cannot be directly obtained via Littlewood-Paley theory, since it breaks down in the $L^\infty$ framework. Following the argument in \cite{GLNY}, we can prove it in section \ref{boundary section} by combining interpolation method, $TT^*$-duality principle, and the decay estimates established in Theorem \ref{ResultTime}.

 Using the decay estimates in Theorem \ref{ResultTime},  
 we establish Strichartz estimates for the Schr\"{o}dinger equation (corresponds to $\phi(r)=r^2$), the wave equation (corresponds to $\phi(r)=r$), the fractional Schr\"{o}dinger equation (corresponds to $\phi(r)=r^\mu$), $0<\mu\neq 1$, the fourth-order Schr\"{o}dinger equation (corresponds to $\phi(r)=r^2+r^4$), the beam equation  (corresponds to $\phi(r)=\sqrt{1+r^4}$), and the Klein-Gordon equation (corresponds to $\phi(r)=\sqrt{1+r^2}$). 

Apart from the introduction, this paper is organized as follows: In Section \ref{sec2}, we review several analytical properties of the operator $\mathcal{L}_a$, including the distorted Fourier transform, the Schr\"odinger kernel and the associated spectral measure. Some auxiliary lemmas are also presented for subsequent proofs of decay and Strichartz estimates. In Section \ref{Proof of Main}, we establish the main result of this paper, namely Theorem \ref{ResultTime}. In section \ref{boundary section}, we prove the boundary Strichartz estimates (i.e., $L_t^q(\mathbb{R},L_{x,y}^\infty(\mathbb{R}^{2+n}))$-estimates)  for the fractional Schr\"{o}dinger equation associated with the operator $\mathcal{L}_a$. Finally, Section \ref{sec5} is dedicated to derive
Strichartz estimates for some concrete wave equations associated with the operator $\mathcal{L}_a$. \\\\

\noindent \textbf{Notations:}
\begin{itemize}
\item  $A\lesssim B$ means $A\leqslant C B$, and $A\sim B$ stands for $C_1B\leq A\leq C_2B$, where $C$, $C_1$, $C_2$ denote positive universal constants.
\item  For any $x, x'\in\mathbb{R}^2$, we express them in polar coordinates as $x=(r\cos \theta,r\sin\theta)$ and $x'=(r'\cos \theta',r'\sin\theta')$.
\item $J_\nu$ denotes the Bessel
function of order $\nu.$ 
\item $B^s_{p,r,a}$ denotes the inhomogeneous Besov space associated with the operator $\mathcal{L}_a$.
\item $\dot{B}^s_{p,r,a}$ denotes the homogeneous Besov space associated  with the operator $\mathcal{L}_a$.
\item $H^s_a$ denotes the inhomogeneous Sobolev space associated with the operator $\mathcal{L}_a$.
\item $\dot{H}^s_a$ denotes the homogeneous Sobolev space associated with the operator $\mathcal{L}_a$.
\end{itemize}

\section{Preliminaries}\label{sec2}
In this section, we recall some analytical features of the operator $\mathcal{L}_a$ and useful technical lemmas. We begin by introducing the concept of the distorted Fourier transform, based on the spectral decomposition of the operator $\mathcal{L}_a$. We refer the reader to \cite{WXZZ, Zhang-Zhang}. 
\subsection{Distorted Fourier Transform}
In this paper, we consider the self-adjoint Schrödinger operator
\begin{equation}\label{6.2}
\mathcal{L}_a=-\Delta_x - \Delta_y + \frac{a}{2} |x|^{-2}, \quad (x,y) \in \mathbb{R}^2 \times \mathbb{R}^n, \:a > 0,  
\end{equation}
where the partial Hardy potential $V(x, y) = \frac{a}{2} |x|^{-2}$ is singular and scaling critical. 

Notice that the potential $V(x, y) = \frac{a}{2} |x|^{-2}$ is radial only with respect to the first two variables $x = (x^1, x^2)$ and independent of the last variable $y$. Thus, the strategy is to utilize the Hankel transform in the variable $r=|x|$ and the classical Fourier transform in the variable $y$. In cylindrical coordinates,
\begin{equation*}
		x_1 = r \cos\theta, \quad x_2 = r \sin\theta, \quad y = y \in \mathbb{R}^n,
\end{equation*}
the operator $\mathcal{L}_a$ can be rewritten as
\begin{equation*}
\mathcal{L}_a=-\partial _{r}^{2}-\frac{1}{r}\partial _r+\frac{-\partial _{\theta}^{2}+\frac{a}{2}}{r^2}-\Delta_y,
\end{equation*}
We choose a complete orthonormal basis $\{\varphi_k\}_{k\in\mathbb{Z}}$ in  $L^2(\mathbb{S}^1)$, where 
\begin{equation*}
\varphi_k(\theta) = \frac{e^{ik\theta}}{\sqrt{2\pi}},
\end{equation*}
which satisfies
\begin{equation*}
\left(-\partial_\theta^2 + \frac{a}{2}\right)\varphi_k(\theta) = \left(\frac{a}{2} + k^2\right)\varphi_k(\theta)=\nu^2 \varphi_k(\theta),
\end{equation*}
with $\nu = \nu(k)= \sqrt{\frac{a}{2} + k^2}$ for any $k \in \mathbb{Z}$. Then for any $f\in L^2(\mathbb{R}^{2+n})$, we can expand $f$ into Fourier series of the form
\begin{equation*}
f(x,y) =f( r,\theta ,y)=\sum_{k\in\mathbb{Z}}a_k( r,y)\varphi_k(\theta) 
\end{equation*}
where 
\begin{equation*}
a_k(r,y)=\int_0^{2\pi}f( r,\theta ,y) \overline{\varphi _k(\theta) }d\theta.
\end{equation*}

For any $f\in L^2(\mathbb{R}^{2+n})$, the Hankel transform of order $\nu\geq0$ is defined by
\begin{equation}\label{6.5}
	(\mathcal{H}_\nu f)(\rho, \theta, y) = \int_0^\infty J_\nu(r\rho) f(r,\theta,y) r dr, 
\end{equation}
where $J_\nu$ is the Bessel function defined by
\begin{equation*}
	 J_\nu(r) = \frac{(r/2)^\nu}{\Gamma(\nu + 1/2)\Gamma(1/2)} \int_{-1}^1 e^{isr} (1 - s^2)^{\nu - 1/2} ds,
\end{equation*}
and the distorted Fourier transform associated with the operator $\mathcal{L}_a$ is defined by
\begin{equation}\label{Distorted}
	\mathcal{F}(f)(\rho, \theta, \eta) = \tilde{f}(\rho, \theta, \eta) = \sum_{k \in \mathbb{Z}} \left( \mathcal{H}_{\nu(k)} \hat{a}_k \right)(\rho, \eta) \varphi_k(\theta),
\end{equation}
where 
\begin{equation*}
\hat{a}_k(r,\eta) =\frac{1}{(2\pi)^{n/2}} \int_{\mathbb{R}^n} e^{-iy \cdot \eta} a_k(r,y) \, dy,
\end{equation*}
is the classical partial Fourier transform with respect to the $y$-variable. 

We need to list the properties of the Hankel transform (see \cite{burq2003}).
\begin{lemma}[Hankel transform]\label{L6.1}
 Let $\mathcal{H}_\nu$ be the Hankel transform in \eqref{6.5} and $A_\nu =-\partial_r^2 - \frac{1}{r}\partial_r + \frac{\nu^2}{r^2}$. Then
\begin{itemize}
\item[(1)] $\mathcal{H}_{\nu} = \mathcal{H}_{\nu}^{-1}$;
\item[(2)] $\mathcal{H}_{\nu}$ is self-adjoint, i.e.,  $\mathcal{H}_{\nu} = \mathcal{H}_{\nu}^{*}$;
\item[(3)] $\mathcal{H}_{\nu}$ is an $L^2$-isometry, i.e., $\| \mathcal{H}_{\nu} f \|_{L^2(\mathbb{R}^{2+n})} = \| f \|_{L^2(\mathbb{R}^{2+n})}$;
\item[(4)] $\mathcal{H}_{\nu}(A_{\nu}f)(\rho,\theta,y) = \rho^2(\mathcal{H}_{\nu}f)(\rho,\theta,y)$, for any $f \in L^2(\mathbb{R}^{2+n})$.
\end{itemize}
\end{lemma}
\noindent From Lemma \ref{L6.1}, for any $g(\rho, \theta, \eta) = \sum\limits_{k \in \mathbb{Z}} b_k(\rho, \eta) \varphi_k(\theta)$, we have the inverse distorted Fourier transform
\begin{equation}
	\mathcal{F}^{-1}(g)(r, \theta, y) = \sum_{k \in \mathbb{Z}} \left( \mathcal{H}_{\nu(k)} \check{b}_k \right)(r, y) \varphi_k(\theta),
\end{equation}
where 
\begin{equation*}
\check{b}_k(\rho, y) =\frac{1}{(2\pi)^{n/2}} \int_{\mathbb{R}^n} e^{iy \cdot \eta} b_k(\rho,\eta) \, d\eta,
\end{equation*}
is the classical partial inverse Fourier transform with respect to the $\eta$-variable. We also have the distorted Plancherel formula
\begin{equation}\label{Plancherel}
\|f\|_{L^2(\mathbb{R}^{2+n})}=\|\mathcal{F}(f)\|_{L^2(\mathbb{R}^{2+n})}. 
\end{equation}

From Lemma \ref{L6.1}, we have
\begin{equation*}
	\mathcal{F}(\mathcal{L}_a f)(\rho, \theta, \eta) = (\rho^2 + |\eta|^2)\mathcal{F}(f)(\rho, \theta, \eta).
\end{equation*}
Therefore, for any Borel measure function $F$, we obtain the functional calculus
\begin{align*}
	F(\mathcal{L} _a)f&=\mathcal{F} ^{-1}\bigl( F(\rho ^2+|\eta |^2)\mathcal{F} (f) \bigr)\\
	&=\int_0^{\infty}{\int_0^{2\pi}{\int_{\mathbb{R} ^n}{K(r,\theta ,y;r^{\prime},\theta ^{\prime},y^{\prime})f(r^{\prime},\theta ^{\prime},y^{\prime})r^{\prime}dr^{\prime}d\theta ^{\prime}dy^{\prime}}}}, 
\end{align*}
where the kernel
\begin{align*}
		K(x, y; x', y')&=K(r, \theta, y; r', \theta', y')\\
		&=\frac{1}{(2\pi)^n}\sum_{k \in \mathbb{Z}} \varphi_k(\theta)\overline{\varphi_k(\theta')} \int_0^\infty \int_{\mathbb{R}^n} e^{i(y - y')\cdot\eta} F(\rho^2 + |\eta|^2) J_{\nu(k)}(r\rho) J_{\nu(k)}(r'\rho) \rho d\rho d\eta,
\end{align*}
with $x=(r\cos \theta,r\sin\theta)$ and $x'=(r'\cos \theta',r'\sin\theta')$.

In particular, take \( F(\tau) = e^{-z\tau} \) with \( \text{Re}(z) > 0 \). Using 
\begin{equation*}
\int_{\mathbb{R} ^n}{e^{-ix\cdot y}e^{-a\left| x \right|^2}dx=\left( \frac{\pi}{a} \right) ^{\frac{n}{2}}e^{-\frac{\left| y \right|^2}{4a}}},\;a>0,
\end{equation*}
and the Weber identity (see \cite{taylor1996}): for $r,r'>0$ and $\nu\geq0$,
\begin{equation*}
\int_0^{\infty}e^{-z\rho ^2}J_\nu(r\rho )J_\nu(r^{\prime} \rho )\rho d\rho =\frac{1}{2z}e^{-\frac{r^2+{r^{\prime}}^2}{4z}}I_\nu\left( \frac{rr^{\prime}}{2z} \right), 
\end{equation*}
where the modified Bessel function $I_\nu(z)$ is given by
\begin{equation*}
	I_\nu(z) = \frac{1}{\pi} \int_0^\pi e^{z\cos s} \cos(\nu s) ds - \frac{\sin \nu \pi}{\pi} \int_0^\infty e^{-z\cosh s} e^{-sv} ds,
\end{equation*}
the corresponding kernel is
\begin{equation*}
K(z; x,y;x',y')=K(z; r, \theta, y; r^{\prime}, \theta^{\prime}, y^{\prime}) 
=\frac{1}{(4\pi z)^\frac{n+2}{2}}e^{-\frac{|y-y^{\prime}|^2}{4z}}e^{-\frac{r^2+{r^{\prime}}^2}{4z}}\sum_{k\in \mathbb{Z}}e^{ik(\theta -\theta ^{\prime})}I_{\nu (k)}\left( \frac{rr^{\prime}}{2z} \right).
\end{equation*}
Letting $z=t>0$, suppose that $e^{-t\mathcal{L}_{a}}(x,y;x',y')=e^{-t\mathcal{L}_{a}}(r,\theta,y;r',\theta',y')$ is the heat kernel of $e^{-t\mathcal{L}}$, then we have the Gaussian bounds (see [Proposition 3.1, \citenum{Zhang-Zhang}]) 
\begin{equation*}
|e^{-t\mathcal{L}_a}(x,y;x',y')|\leq C|t|^{-\frac{n+2}{2}}e^{-\frac{|(x,y)-(x',y')|^2}{8t}}.
\end{equation*}

\subsection{The Schr\"odinger kernel and spectral measure associated with the operator $\mathcal{L}_a$}
Zhang-Zhang [Proposition 3.5, \citenum{Zhang-Zhang}] constructed the representation of the Schrödinger kernel and derived the dispersive estimates for the Schr\"odinger propagator $e^{-it\mathcal{L}_a}$.
\begin{proposition}\label{Prop 6.2}
Let \( \mathcal{L}_a \) be the operator given in (\ref{6.2}) and let \( x, x^{\prime} \in \mathbb{R}^2 \) and \( y, y' \in \mathbb{R}^n \). Suppose \( K_a(t; x, y, x', y') \) is the kernel of \( e^{-it\mathcal{L}_a} \). Define
\begin{equation}\label{A-B}
\begin{aligned}
A_a(s; \theta, \theta') &=\sum_{k\in \mathbb{Z}}e^{ik(\theta -\theta ')}\left( \cos(\sqrt{k^2+c}\,s)-\cos(|k|s) \right)\\
                        &= \left( \cos(\sqrt{c} s) - 1 \right) + sG_1(s; \theta, \theta') + s^2G_2(s; \theta, \theta'), \\
B_a(s; \theta, \theta') &=\sum_{k \in \mathbb{Z}} e^{ik(\theta - \theta')} \sin(\sqrt{k^2 + c}\pi) e^{-s\sqrt{k^2 + c}}\\
                        &= \sin(\sqrt{c} \pi) \, e^{-s\sqrt{c}} + D_1(s; \theta, \theta') + D_2(s; \theta, \theta'),
\end{aligned}
\end{equation}
where $c=\frac{a}{2}$, $G_1,G_2$ satisfy
\begin{equation}\label{G1-G2}
| G_1(s; \theta, \theta') | + | G_2(s; \theta, \theta') | \leq C, \quad 0 \leq s \leq \pi, 
\end{equation}
and $D_1, D_2$ satisfy
\begin{equation}\label{D1-D2}
	|D_1(s; \theta, \theta')| + |D_2(s; \theta, \theta')| \leq C\Bigl( e^{-s/2} + \Bigl| \sum_{\substack{k \in \mathbb{Z} \\ k \neq 0}} \frac{e^{ik(\theta - \theta')}(-1)^{|k|}}{|k|} e^{-s|k|} \Bigr| \Bigr).
\end{equation}
Then we have
\begin{equation*}
K_a(t; x, y; x', y') = K_0(t; x, y; x', y') + E_a(t; x, y; x', y'),
\end{equation*}
where 
\begin{equation*}
K_0(t; x, y; x', y') = C_n|t|^{-\frac{2+n}{2}}e^{-\frac{|(x,y)-(x',y')|^2}{4it}},
\end{equation*}
with $C_n$ being a constant only depending on $n$ and 
\begin{equation*}
\begin{aligned}
E_a(t;x,y;x',y') &= \frac{C_n}{|t|^{\frac{2+n}{2}}} e^{-\frac{|y-y'|^2}{4it}} 
\Bigg( \int_{|\theta-\theta'|}^{\pi} e^{-\frac{r^2 + r'^2 - 2rr'\cos s}{4it}} A_a(s;\theta,\theta') \, ds \\
&\quad - \int_0^\infty e^{-\frac{r^2 + r'^2 - 2rr'\cosh s}{4it}} B_a(s;\theta,\theta') \, ds \Bigg).
\end{aligned}
\end{equation*}
Furthermore, there exists a positive constant $C$ such that
\begin{equation*}
|E_a(t; x, y; x', y')| \leq C |t|^{-\frac{2+n}{2}}.
\end{equation*}
\end{proposition}
\begin{remark}
It can be derived from \eqref{A-B}, \eqref{G1-G2} and \eqref{D1-D2} that
\begin{equation}\label{6.20}
\int_{0}^{\pi} |A_a(s; \theta, \theta')| \, ds + \int_{0}^{\infty} |B_a(\theta, \theta', s)| \, ds \leq C,
\end{equation}
\end{remark}
Since the half-wave kernel can not be written in an explicit form, establishing Strichartz estimates for the wave equation turns out to be much more complicated. To resolve this obstacle, Zhang-Zhang \cite{Zhang-Zhang} constructed an explicit representation for the spectral measure associated with the operator $\mathcal{L}_a$.
\begin{proposition}\label{Prop 6.3}
Let $\mathcal{L}_a$ be the operator given in \eqref{6.2} and let $x,x'\in \mathbb{R}^2$ and $y, y' \in \mathbb{R}^n$. Let $dE_{\sqrt{\mathcal{L}_a}}(\lambda; x, y; x', y')$ be the kernel of the spectral measure, then
\begin{equation}\label{spectral}
\begin{aligned}
dE_{\sqrt{\mathcal{L}_a}}(\lambda; x, y; x', y') &= \frac{\lambda^{n+1}}{\pi^{n+1}} \sum_{\pm} a_{\pm}(\lambda|(x, y) - (x', y')|) e^{\pm i\lambda|(x, y) - (x', y')|}  \\
 &+ \frac{\lambda^{n+1}}{\pi^{n+1}} \int_{|\theta - \theta'|}^{\pi} \sum_{\pm} a_{\pm}(\lambda|\vec{n}_s^1|) e^{\pm i\lambda|\vec{n}_s^1|} \times A_a(s; \theta, \theta') ds  \\
&- \frac{\lambda^{n+1}}{\pi^{n+1}} \int_0^\infty \sum_{\pm} a_{\pm}(\lambda|\vec{n}_s^2|) e^{\pm i\lambda|\vec{n}_s^2|} \times B_a(s, \theta, \theta') ds, 
\end{aligned}
\end{equation}
where $a_{\pm}(r)$ satisfies
\begin{equation*}
|\partial_\lambda^k a_\pm(r)| \leq C_k (1 + r)^{-\frac{n+1}{2}-k}, \;\forall k\in\mathbb{N},
\end{equation*}
and $A_a(s, \theta, \theta'), B_a(s, \theta, \theta')$ are as \eqref{A-B},
and
\begin{align*}
\vec{n}_s^1 &= (r - r', \sqrt{2rr'(1-\cos s)}, y - y'), \\
\vec{n}_s^2 &= (r + r', \sqrt{2rr'(\cosh s - 1)}, y - y').
\end{align*}
\end{proposition}
\begin{remark}
The functions $a_\pm(r)$ satisfy
\begin{equation*}
\int_{\mathbb{S}^{n+1}} e^{-ix \cdot \omega} \, d\sigma(\omega) = \sum_{\pm} a_{\pm}(|x|) e^{\pm i|x|}.
\end{equation*}
By the vanishing property of $a_\pm(r)$ at the origin and infinity, we have
\begin{equation}\label{6.19}
|\partial_\lambda^k(a_\pm(\lambda r))| \leq C_k\min\{\lambda^{-k} (1 + \lambda r)^{-\frac{n+1}{2}}, r^k\}, \;\forall \lambda>0 \text{ and } k\in\mathbb{N}.
\end{equation}
\end{remark}

\subsection{Technical Lemmas} 
From the spectral measure associated with the operator $\mathcal{L}_a$ and Young's inequality, the $L^1(\mathbb{R}^{2+n})\to L^\infty(\mathbb{R}^{2+n})$ decay estimates in Theorem \ref{ResultTime} are reduced to a list of oscillatory integrals. In order to estimate the oscillatory integrals, we introduce the stationary phase lemma.
\begin{lemma}[van der Corput lemma \cite{S1993}]\label{phase} Let $g\in C^\infty([a,b])$ be real-valued such that $|g''(x)|\geq \delta $ for any $x\in[a,b]$ with $\delta >0$. Then for any function $\psi \in C^\infty([a,b])$, there exists a constant $C$ (does not depend on $\delta, a, b, g$ or $\psi$)  such that
	\begin{equation*}
	\left|\int_a^b e^{{ ig(x)}}\psi(x)\,dx\right|\leq C{ \delta ^{-1/2}}\left(\|\psi\|_\infty+\|\psi'\|_1\right).
	\end{equation*}
 \end{lemma}
We also exploit the following estimates, which can be easily proved by comparing the sums with the corresponding integrals.
\begin{lemma}\label{Sum}
Fix $\beta>0$. There exists $C_\beta>0$ such that for any $A>0$, we have
\begin{align*}
\sum_{j\in\mathbb{Z}, 2^j\leq A}2^{j\beta}&\leq C_\beta A^\beta,\\
\sum_{j\in\mathbb{Z}, 2^j>A}2^{-j\beta}&\leq C_\beta A^{-\beta}.
\end{align*}
\end{lemma}
Finally, we apply the following duality arguments (see \cite{GV}) to prove the Strichartz estimates for some concrete equations related to the operator $\mathcal{L}_a$.
\begin{lemma}\label{equ} Let $H$ be a Hilbert space, $X$ be a Banach space, $X^*$ be the dual of $X$, and $D$ be a dense vector space contained in $X$. Let $A\in \mathcal{L}_a(D,H)$ and   $A^*\in \mathcal{L}_a(H,D_a^*)$ be its adjoint operator, defined by
	\begin{equation*}
	\langle A^*v,f \rangle_D=\langle v,Af \rangle_H, \quad \forall f\in D, \quad \forall v \in H,
	\end{equation*}
where $\mathcal{L}_a(Y,Z)$ is the space of linear maps from a vector space $Y$ to a vector space $Z$, $D^*_a$ is the algebraic dual of $D$, $\langle \varphi,f\rangle_D$ is the pairing between $D^*_a$ and $D$, and $\langle\cdot,\cdot\rangle_H$ is the scalar product in $H$. Then the following three conditions are equivalent:
\begin{enumerate}
    \item There exists  $0\leqslant a < \infty$ such that for all $f \in D$,
	\begin{equation*}
	\|Af\| \leqslant a\|f\|_X.
	\end{equation*}
 \item  $\mathfrak{R}(A^*)\subset X^*$, and there exists $0\leqslant a <\infty$, such that for all $v \in H$,
	\begin{equation*}
	\|A^*v\|_{X^*} \leqslant a\|v\|.
	\end{equation*}
 \item $\mathfrak{R}(A^*A)\subset X^*$, and there exists $a$, $0\leqslant a < \infty$, such that for all $f \in D$,
	\begin{equation*}
	\|A^*Af\|_{X^*} \leqslant a^2\|f\|_X,
	\end{equation*}
\end{enumerate}
where $\|\cdot\|$ denote the norm in $H$ and the constant $a$ is the same in all three parts. If one of those conditions is satisfied, the operators $A$ and $A^*A$ extend by continuity to bounded operators from $X$ to $H$ and from $X$ to $X^*$, respectively.
\end{lemma}
\begin{lemma} \label{DuilatyXX} Let $H$, $D$ and two triplets $(X_i,A_i,a_i), i=1,2$, satisfy any of the conditions of Lemma \ref{equ}. Then for all choices of $i,j=1,2$, $\mathfrak{R}(A_i^*A_j) \subset X_i^*$, and for all $f \in D$, we have
	\begin{equation*}
	\|A_i^* A_jf\|_{X_i^*} \leqslant a_i a_j\|f\|_{X_j}.
	\end{equation*}
\end{lemma}
\begin{lemma}\label{L1LW}  Let $H$ be a Hilbert space, $I$ be an interval of $\mathbb{R}$,    $X \subset \mathcal{S}'(I \times \mathbb{R}^n)$ be a Banach space which is stable under time restriction, and let $X$ and $A$ satisfy (any of) the conditions of Lemma \ref{equ}. Then the operator $A^*A$ is a bounded operator from $L_t^1(I,H)$ to $X^*$ and from $X$ to $L_t^\infty(I,H)$.
\end{lemma}
\section{Proof of main results}\label{Proof of Main}
This section is devoted to proving Theorem \ref{ResultTime}, i.e., the decay estimate for a class of
dispersive semigroups $U_k(t)=e^{it\phi(\sqrt{\mathcal{L}_a})}\psi(2^{-k}\sqrt{\mathcal{L}_a})$ on $\mathbb{R}^{2+n}$.
 \begin{proof}[Proof of Theorem \ref{ResultTime}]
We will use the stationary phase argument to prove Theorem \ref{ResultTime}. Above all, we prove part \textbf{(B)}. Note that 
\begin{equation*}
U_k(t)f(x,y)=\int_{\mathbb{R}^{2+n}}\int_0^\infty e^{it\phi(\lambda)} \psi(2^{-k}\lambda)dE_{\sqrt{\mathcal{L}_a}}(\lambda; x,y; x',y')f(x',y')dx'dy',
\end{equation*}
whose kernel $I_k(x,y;x',y')$ is denoted by 
\begin{equation*}
I_k(x,y;x',y')=\int_0^\infty e^{it\phi(\lambda)} \psi(2^{-k}\lambda)dE_{\sqrt{\mathcal{L}_a}}(\lambda; x,y; x',y'),
\end{equation*}
and it suffices to prove the following statements: for $k \leq 0$, assume that $\phi$ satisfies (C2), then
	\begin{equation*}
	\|I_k\|_{L^\infty(\mathbb{R}^{2+n}\times\mathbb{R}^{2+n})} \lesssim |t|^{-\theta}2^{k\left(n+2-m_2\theta\right)},\,0 \leq \theta \leq \frac{n+1}{2};
	\end{equation*}
 In addition, if $\phi$ satisfies (C4), then
	\begin{equation*}
	\|I_k\|_{L^\infty(\mathbb{R}^{2+n}\times\mathbb{R}^{2+n})} \lesssim |t|^{-\frac{n+1+\theta}{2}}2^{k\left(n+2-\frac{m_2(n+1+\theta)}{2}-\frac{\theta(\alpha_2-m_2)}{2}\right)},\,   0 \leq \theta \leq 1.
	\end{equation*}
Indeed, by \eqref{spectral}, let us divide the kernel $I_k$ into three parts by
\begin{equation*}
I_k(x,y;x',y')=\frac{1}{\pi^{n+1}}\sum_{\pm}\left(I_k^{\pm,1}(x-x',y-y')+I_k^{\pm,2}(x,y;x',y')+I_k^{\pm,3}(x,y;x',y')\right), 
\end{equation*}
where
\begin{equation*}
I_k^{\pm,1}(x,y)=\int_0^\infty e^{it\phi(\lambda)} \psi(2^{-k}\lambda)\lambda^{n+1}a_{\pm}(\lambda|(x,y)|)e^{\pm i\lambda|(x,y)|}d\lambda,
\end{equation*}
\begin{equation*}
I_k^{\pm,2}(x,y;x',y')=\int_{|\theta-\theta'|}^\pi \int_0^\infty  e^{it\phi(\lambda)} \psi(2^{-k}\lambda)\lambda^{n+1}a_{\pm}(\lambda|\overset{\rightarrow}{n}_s^1|)e^{\pm i\lambda|\overset{\rightarrow}{n}_s^1|}d\lambda\; A_a(s;\theta,\theta')ds,
\end{equation*}
and
\begin{equation*}
I_k^{\pm,3}(x,y;x',y')=\int_0^\infty \int_0^\infty  e^{it\phi(\lambda)} \psi(2^{-k}\lambda)\lambda^{n+1}a_{\pm}(\lambda|\overset{\rightarrow}{n}_s^2|)e^{\pm i\lambda|\overset{\rightarrow}{n}_s^2|}d\lambda\; B_a(s;\theta,\theta')ds.
\end{equation*}
Then our aim becomes to prove the following statements: for $k \leq 0$, assume that $\phi$ satisfies (C2), then for $\ell=2,3$ and $0 \leq \theta \leq \frac{n+1}{2}$
	\begin{align}
	\|I^{\pm,1}_k\|_{L^\infty(\mathbb{R}^{2+n})} &\lesssim |t|^{-\theta}2^{k\left(n+2-m_2\theta\right)},\label{1a}\\
    \|I^{\pm,\ell}_k\|_{L^\infty(\mathbb{R}^{2+n}\times\mathbb{R}^{2+n})} &\lesssim |t|^{-\theta}2^{k\left(n+2-m_2\theta\right)}. \label{ia}
	\end{align}
 In addition, if $\phi$ satisfies (C4), then for $\ell=2,3$ and $0 \leq \theta \leq 1$
 \begin{align}
	\|I^{\pm,1}_k\|_{L^\infty(\mathbb{R}^{2+n})} &\lesssim |t|^{-\frac{n+1+\theta}{2}}2^{k\left(n+2-\frac{m_2(n+1+\theta)}{2}-\frac{\theta(\alpha_2-m_2)}{2}\right)},\label{1aa}\\
    \|I^{\pm,\ell}_k\|_{L^\infty(\mathbb{R}^{2+n}\times\mathbb{R}^{2+n})} &\lesssim |t|^{-\frac{n+1+\theta}{2}}2^{k\left(n+2-\frac{m_2(n+1+\theta)}{2}-\frac{\theta(\alpha_2-m_2)}{2}\right)}. \label{iaa}
	\end{align}
Next, it suffices to prove \eqref{1a} and \eqref{1aa}. Indeed, once these two estimates are established, \eqref{ia} and \eqref{iaa} follow by replacing $|(x,y)|$ by $|\overset{\rightarrow}{n}_s^\ell|, \ell=2,3$ and invoking \eqref{6.20}, together with arguments analogous to those used for \eqref{1a} and \eqref{1aa}. To prove  \eqref{1a} and \eqref{1aa},
we rewrite the term $I_k^{\pm,1}$ by
\begin{equation*}
I_k^{\pm,1}(x,y)=2^{k(n+2)}\mathcal{I}_k^{\pm,1}(x,y)
\end{equation*}
where
\begin{equation*}
\mathcal{I}_k^{\pm,1}(x,y)=\int_0^\infty e^{it\phi(2^k\lambda)} \psi(\lambda)\lambda^{n+1}a_{\pm}(2^k\lambda|(x,y)|)e^{\pm i2^k\lambda|(x,y)|}d\lambda,
\end{equation*}
and it is equivalent to prove the following statements: 
for $k \leq 0$, assume that $\phi$ satisfies (C2), then for $0 \leq \theta \leq \frac{n+1}{2}$
	\begin{equation}
	\|\mathcal{I}^{\pm,1}_k\|_{L^\infty(\mathbb{R}^{2+n})}\lesssim \left(2^{km_2}|t|\right)^{-\theta}\label{11a},
	\end{equation}
 In addition, if $\phi$ satisfies (C4), then $0 \leq \theta \leq 1$
 \begin{equation}
	\|\mathcal{I}^{\pm,1}_k\|_{L^\infty(\mathbb{R}^{2+n})} \lesssim  \left(2^{km_2}|t|\right)^{-\frac{n+1+\theta}{2}}2^{\frac{k\theta(m_2-\alpha_2)}{2}}.\label{11aa}
	\end{equation}
We will prove the above statements in two cases, where we use the vanishing property of $a_\pm$ at the origin and the infinity, respectively.

\textbf{\underline{Case 1.} $2^k|(x,y)|\leq 1$:} In this case, we shall use the vanishing property of $a_\pm$ at the origin. 
By denoting $D_\lambda=\frac{1}{i2^kt\phi'(2^k\lambda)}\frac{d}{d\lambda}$, we notice that
\begin{equation*}
D_\lambda\left(e^{it\phi(2^k\lambda)}\right)=e^{it\phi(2^k\lambda)}, \; D_\lambda^*f=\frac{i}{2^kt}\frac{d}{d\lambda}\left(\frac{1}{\phi'(2^k\lambda)}f\right).
\end{equation*}
For $j \in\mathbb{N}$ and $\lambda\in [1/2,1]$, it follows from the assumption (C2) that
\begin{equation}\label{phase-prime}
 \frac{d^j}{d\lambda^j}\left(\frac{1}{\phi'(2^k\lambda)}\right)\leq C_j 2^{-k(m_2-1)}.
\end{equation}
For any $Q\in \mathbb N$, using integration by part $Q$-times, we obtain
 \begin{equation*}
	\begin{aligned}
	\mathcal{I}_k^{\pm,1}(x,y)&=\int_0^\infty e^{it\phi(2^k\lambda)} \psi(\lambda)\lambda^{n+1}a_{\pm}(2^k\lambda|(x,y)|)e^{\pm i2^k\lambda|(x,y)|}d\lambda\\
    &=\int_0^\infty D^Q_\lambda\left(e^{it\phi(2^k\lambda)}\right)\psi(\lambda)\lambda^{n+1}a_{\pm}(2^k\lambda|(x,y)|)e^{\pm i2^k\lambda|(x,y)|}d\lambda\\
    &=\int_0^\infty e^{it\phi(2^k\lambda)}\left(D_\lambda^*\right)^Q\left(\psi(\lambda)\lambda^{n+1}a_{\pm}(2^k\lambda|(x,y)|)e^{\pm i2^k\lambda|(x,y)|}\right)d\lambda\\
    &=\left(\frac{i}{2^kt}\right)^{ Q}\sum_{j=0}^Q\sum_{\beta\in \Lambda_Q^j} C_{Q,j}\int_0^\infty  e^{it\phi(2^k\lambda)}\\
     &\qquad \qquad\times \prod_{i=1}^Q\frac{d^{\beta_i}}{d\lambda^{\beta_i}}\left(\frac{1}{\phi'(2^k\lambda)}\right)\frac{d^{Q-j}}{d\lambda^{Q-j}}\left(\psi(\lambda)\lambda^{n+1}a_{\pm}(2^k\lambda|(x,y)|)e^{\pm i2^k\lambda|(x,y)|} \right)\,d\lambda, 
    \end{aligned}
	\end{equation*}
 where $\Lambda_Q^j=\{\beta=(\beta_1,\beta_2,\cdots,\beta_Q)\in\{0,1,2,\cdots,j\}^Q: \sum\limits_{i=1}^Q\beta_i=j\}$. Now using the estimates \eqref{6.19} and \eqref{phase-prime}
 \begin{equation}\label{small}
|\mathcal{I}_k^{\pm,1}(x,y)|\leq C_Q \left(2^k|t|\right)^{-Q}2^{-kQ(m_2-1)}=C_Q\left(2^{km_2}|t|\right)^{-Q}, 
 \end{equation} for any $Q\in \mathbb N$. Using the Riesz-Thorin interpolation between different $Q\in\mathbb{N}$ in \eqref{small}, for any $\theta\geq0$, we have
\begin{equation}\label{sharp-s-small}
|\mathcal{I}_k^{\pm,1}(x,y)|\leq C_\theta\left(2^{km_2}|t|\right)^{-\theta}.
 \end{equation}  
 This completes the proof of \eqref{11a} in Case 1.   
 
\textbf{\underline{Case 2.}} $2^k|(x,y)|>1$. In this case, we apply the decay property of $a_\pm$ at the infinity. We can rewrite
\begin{align*}
\mathcal{I}_k^{\pm,1}(x,y)&=\int_0^\infty e^{it\phi(2^k\lambda)} \psi(\lambda)\lambda^{n+1}a_{\pm}(2^k\lambda|(x,y)|)e^{\pm i2^k\lambda|(x,y)|}d\lambda\\
   &=\int_0^\infty  e^{i\left(t\phi(2^k\lambda)\pm 2^k\lambda|(x,y)|\right)}\psi(\lambda)\lambda^{n+1}a_{\pm}(2^k\lambda|(x,y)|)d\lambda\\
   &=\int_0^\infty   e^{it\phi_\pm(2^k\lambda)}\psi(\lambda)\lambda^{n+1}a_{\pm}(2^k\lambda|(x,y)|)d\lambda,
\end{align*}
where $\phi_\pm(\lambda)=\phi(\lambda)\pm \frac{\lambda |(x,y)|}{t}$.
Without loss of generality, we can assume that $t>0$ and $\phi'(r)>0$. Now we will estimate $\mathcal{I}_k^{+,1}(x,y)$ and $\mathcal{I}_k^{-,1}(x,y)$ separately.

For $\mathcal{I}_k^{+,1}(x,y)$, note that $\phi'_+(\lambda)\geq \phi'(\lambda)$ and \eqref{phase-prime} also holds true if we replace $\phi$ by $\phi_+$. Then, proceeding similarly as Case 1, combined with \eqref{6.19},  for any 
 $\theta\geq0,$ we obtain
\begin{equation}\label{sharp-s-big-plus}
|\mathcal{I}_k^{+,1}(x,y)|\leq C_\theta\left(2^{km_2}|t|\right)^{-\theta}.
 \end{equation}
 For $\mathcal{I}_k^{-,1}(x,y)$, $\phi'_-(2^k\lambda)=\phi'(2^k\lambda)-\frac{|(x,y)|}{t}$. Noticing that if $|(x,y)|=t\phi'(2^k\lambda)$, then $\phi'_-(2^k\lambda)=0$. We continue our discussion by considering the following three sub-cases:
\begin{itemize}
\item When $|(x,y)|\geq 2t\sup\limits_{\lambda\in[1/2,2]}\phi'(2^k\lambda)$.
 In this case, we have $|\phi'_-(2^k\lambda)|=\frac{1}{t}\left(|(x,y)|-t\phi'(2^k\lambda)\right)\geq \phi'(2^k\lambda)$ and  the estimate 
 \eqref{phase-prime} also holds true if we replace $\phi$ by $\phi_-$. Then, analogous to $\mathcal{I}_k^{+,1}(x,y)$, for any $\theta\geq0$, we obtain
\begin{equation}\label{sharp-s-big-minus1}
|\mathcal{I}_k^{-,1}(x,y)|\leq C_\theta\left(2^{km_2}|t|\right)^{-\theta}.
 \end{equation}
\item  When $|(x,y)|\leq \frac{1}{2}t\inf\limits_{\lambda\in[1/2,2]}\phi'(2^k\lambda)$. In this case, we have $\phi'_-(2^k\lambda)=\frac{1}{t}\left(t\phi'(2^k\lambda)-|(x,y)|\right)\geq \frac{1}{2}\phi'(2^k\lambda)$ and the estimate \eqref{phase-prime} also holds true if we replace $\phi$ by $\phi_-$. Analogous to $\mathcal{I}_k^{+,1}(x,y)$, for any $\theta\geq0$, we obtain
\begin{equation}\label{sharp-s-big-minus2}
|\mathcal{I}_k^{-,1}(x,y)|\leq C_\theta\left(2^{km_2}|t|\right)^{-\theta}.
 \end{equation}
\item When $\frac{t}{2}\inf\limits_{\lambda\in[1/2,2]}\phi'(2^k\lambda)\leq |(x,y)|\leq 2t\sup\limits_{\lambda\in[1/2,2]}\phi'(2^k\lambda)$. In this case, we see that $|(x,y)|\sim t2^{k(m_2-1)}$. Moreover, it derives from \eqref{6.19} that
 \begin{equation}\label{sharp-s-big-minus3}
  |\mathcal{I}_k^{-,1}(x,y)|\lesssim \min\left(1,\left(2^k|(x,y)|\right)^{-(n+1)/2}\right)\lesssim \min\left(1,\left(2^{km_2}|t|\right)^{-(n+1)/2}\right).
 \end{equation}
Hence, for all $0\leq\theta\leq(n+1)/2$, we get
\begin{equation}\label{sharp}
|\mathcal{I}_k^{-,1}(x,y)|\lesssim \left(2^{km_2}|t|\right)^{-\theta}
 \end{equation}
 \end{itemize}
Combining \eqref{sharp-s-big-plus}, \eqref{sharp-s-big-minus1}, \eqref{sharp-s-big-minus2} and \eqref{sharp}, we complete the proof of \eqref{11a} in Case 2.

If (C4) holds in addition, we have $\left|\frac{d^2}{d\lambda^2}\left(\phi_-(2^k\lambda)\right)\right|=2^{2k}|\phi''(2^k\lambda)|\sim 2^{k\alpha_2}$ in the last sub-case discussed above. It follows from Lemma \ref{phase} and \eqref{6.19} that
\begin{equation}\label{better-sharp-s-big-minus3}
\begin{aligned}
|\mathcal{I}_k^{-,1}(x,y)|&\lesssim |t2^{k\alpha_2}|^{-1/2}\left(2^k|(x,y)|\right)^{-(n+1)/2}\\
&\lesssim |t2^{k\alpha_2}|^{-1/2}\left(2^{km_2}|t|\right)^{-(n+1)/2}\\
&\lesssim \left(2^{km_2}|t|\right)^{-(n+2)/2}2^{k(m_2-\alpha_2)/2}.
\end{aligned}
 \end{equation}
For any $0\leq \theta\leq 1$,   since $(n+1+\theta)/2=(1-\theta)(n+1)/2+\theta (n+2)/2$, by an interpolation between \eqref{sharp} (with $\theta=(n+1)/2$) and \eqref{better-sharp-s-big-minus3}, we obtain
\begin{equation}\label{Better-sharp-s-big-minus3}
|\mathcal{I}_k^{-,1}(x,y)|\lesssim \left(2^{km_2}|t|\right)^{-(n+1+\theta)/2}2^{k\theta(m_2-\alpha_2)/2}.
 \end{equation}
Noting that $\alpha_2\geq m_2$, from \eqref{sharp-s-small}, \eqref{sharp-s-big-plus}, \eqref{sharp-s-big-minus1}, \eqref{sharp-s-big-minus2} and \eqref{Better-sharp-s-big-minus3}, we have
\begin{equation}\label{better-sharp}
|\mathcal{I}_k^{\pm,1}(x,y)|\lesssim \left(2^{km_2}|t|\right)^{-(n+1+\theta)/2}2^{k\theta(m_2-\alpha_2)/2},
 \end{equation}
 for all $0\leq\theta\leq 1$, which completes the proof of \eqref{11aa}.\\

Next, we just need to prove part \textbf{(C)}, since the proof of part \textbf{(A)} is similar to part \textbf{(B)}.

From the results in part \textbf{(B)}, it is easy to see that
\begin{equation*}
\|U^{low}(t)f\|_{L^\infty(\mathbb{R}^{2+n})}\leq \sum_{k\leq 0}\left\|U_k(t)f\right\|_{L^\infty(\mathbb{R}^{2+n})}\lesssim \sum_{k\leq 0}2^{k(2+n)}\left\|f\right\|_{L^1(\mathbb{R}^{2+n})}\lesssim \left\|f\right\|_{L^1(\mathbb{R}^{2+n})}.
\end{equation*}
Therefore, if (C2) holds, to prove \eqref{res-sum1} in part \textbf{(C)}, it is equivalent to get
\begin{equation*}
\|U^{low}(t)f\|_{L^\infty(\mathbb{R}^{2+n})}\lesssim |t|^{-\theta}
\left\|f\right\|_{L^1(\mathbb{R}^{2+n})},\quad \theta=\min\left(\frac{2+n}{m_2},\frac{1+n}{2}\right).
\end{equation*}
First, we consider the case  when $\theta=\frac{1+n}{2}<\frac{2+n}{m_2}$, i.e., $2+n-\frac{(1+n)m_2}{2}>0$. Now applying part \textbf{(B)}, we obtain
 \begin{equation*}
\begin{aligned}
\|U^{low}(t)(t)f\|_{L^\infty(\mathbb{R}^{2+n})}&\leq \sum_{k\leq 0}\left\|U_k(t)f\right\|_{L^\infty(\mathbb{R}^{2+n})}\\
&\lesssim\sum_{k\leq  0}|t|^{-\frac{1+n}{2}}2^{k\left(2+n-\frac{(1+n)m_2}{2}\right)}\left\|f\right\|_{L^1(\mathbb{R}^{2+n})}\\
&\lesssim |t|^{-\frac{1+n}{2}}  \left\|f\right\|_{L^1(\mathbb{R}^{2+n})}.
\end{aligned}
\end{equation*}
Therefore, it suffices to consider the case $\theta=\frac{2+n}{m_2}\leq\frac{1+n}{2}$, i.e., $2+n-m_2(1+n)/2\leq 0$. From the proof of of part \textbf{(B)}, 
\begin{align*}\label{U2}
U^{low}(t)f(x,y)&=\sum_{k\leq 0}U_k(t)f(x,y)\\
&=\int_{\mathbb{R}^{2+n}}I(x,y;x',y')f(x',y')dx'dy',
\end{align*}
with the kernel
\begin{align*}
I(x,y;x',y')&=\sum_{k\leq 0}I_k(x,y;x',y')\\
            &=\frac{1}{\pi^{n+1}}\sum_{k\leq 0}\sum_{\pm}\left(I_k^{\pm,1}(x-x',y-y')+I_k^{\pm,2}(x,y;x',y')+I_k^{\pm,3}(x,y;x',y')\right)\\
            &=\frac{1}{\pi^{n+1}}\sum_{\pm}\left(I^{\pm,1}(x-x',y-y')+I^{\pm,2}(x,y;x',y')+I^{\pm,3}(x,y;x',y')\right),
\end{align*}
where $I^{\pm,\ell}=\sum_{k\leq 0}I_k^{\pm,\ell}$, $\ell=1,2,3$. By Young's inequality, it suffices to prove
\begin{equation*}
|I(x,y;x',y')|\lesssim |t|^{-\frac{2+n}{m_2}},\quad \forall (x,y),(x,y')\in \mathbb{R}^{2+n}.
\end{equation*}
Argued similarly as part {(B)}, we only need to prove
\begin{equation*}
|I^{\pm,1}(x,y)|\lesssim |t|^{-\frac{2+n}{m_2}},\quad \forall (x,y)\in \mathbb{R}^{2+n}.
\end{equation*}
Indeed, we write
\begin{equation*}
I^{\pm,1}(x,y)=\sum_{k\leq0}I_k^{\pm,1}(x,y)=\sum_{k\leq0}2^{k(2+n)}\mathcal{I}_k^{\pm,1}(x,y).
\end{equation*}
From the proof of of part \textbf{(B)}, we know that if $0\geq k_0\in\mathbb{Z} $ such that $2^{k_0}|(x,y)|>1$ and $|(x,y)|\sim t2^{k_0(m_2-1)}$, then
\begin{equation}\label{sim}
 \begin{aligned}
  2^{k_0(2+n)}\left|\mathcal{I}_k^{\pm,1}(x,y)\right|
  &\lesssim 2^{k_0(2+n)}\left(2^{k_0m_2}|t|\right)^{-(1+n)/2}=\\
  &=|t|^{-(1+n)/2}2^{k_0\left(2+n-m_2(1+n)/2\right)}\\
  &\lesssim |t|^{-(1+n)/2}\left(\frac{2^{k_0}|(x,y)|}{|t|}\right)^{\left(2+n-m_2(1+n)/2\right)\frac{1}{m_2}}\\
  &\lesssim |t|^{-\frac{n+2}{m_2}}.
  \end{aligned}
 \end{equation}
When $|k-k_0|>C\gg1$, it holds true that
 \begin{equation}\label{gg}
\left|\mathcal{I}_k^{\pm,1}(x,y)\right|\lesssim\left(2^{km_2}|t|\right)^{-\alpha},\;\forall \alpha\geq0.
 \end{equation}
Hence, by \eqref{sim} and \eqref{gg}, we get
 \begin{equation}\label{last}
 \begin{aligned}
  &\quad \left|I^{\pm,1}(x,y)\right|\\
&\leq\sum_{k\leq0}2^{k(2+n)}\left|\mathcal{I}_k^{\pm,1}(x,y)\right|\\
  &\leq \left(\sum_{|k-k_0|\leq C}+\sum_{|k-k_0|>C,\;2^k\leq|t|^{-\frac{1}{m_2}}}+\sum_{|k-k_0|>C,\;2^k>|t|^{-\frac{1}{m_2}}}\right)2^{k(2+n)}\left|\mathcal{I}_k^{\pm,1}(x,y)\right|\\
  &\lesssim |t|^{-\frac{2+n}{m_2}}+\sum_{2^k\leq|t|^{-\frac{1}{m_2}}}2^{k(2+n)}+\sum_{2^k>|t|^{-\frac{1}{m_2}}}|t|^{-\alpha}2^{k(2+n-m_2\alpha)}\\
  &\lesssim |t|^{-\frac{2+n}{m_2}},
  \end{aligned}
  \end{equation}
 where the last inequality is obtained by applying Lemma \ref{Sum} and choosing $\alpha>\frac{2+n}{m_2}$. Hence, we complete the proof of \eqref{res-sum1} in part \textbf{(C)}.

 In addition, if (C4) holds, we can argue in a similar way as previously.
 \end{proof}

\section{On the boundary Strichartz estimates for the fractional Schr\"odinger equation}\label{boundary section}
This section is devoted to boundary Strichartz estimates for the
fractional Schr\"odinger operator $e^{it\mathcal{L}_a^\nu}$. More precisely, we investigate estimates on the mixed time-space function space $L_t^q(\mathbb{R},L_{x,y}^\infty(\mathbb{R}^{2+n}))$. For convenience, we denote $n_\nu=\begin{cases}1,\nu=\frac{1}{2},\\
0,\nu\neq\frac{1}{2}.\end{cases}$
\begin{theorem}\label{boundary}
Let $\nu>0$, $n\geq 1$ and $2<q<\infty$. Then we have
\begin{equation}\label{boundary estimate}
\|e^{it\mathcal{L}_a^\nu} f\|_{L_t^q(\mathbb{R},L_{x,y}^\infty(\mathbb{R}^{2+n}))}\lesssim \|f\|_{\dot{H}_a^{\frac{n+2}{2}-\frac{2\nu}{q}}},
\end{equation}
if either of the following conditions holds
\begin{itemize}
\item $\nu=\frac{1}{2}$ and $\frac{2}{q}<\frac{n+1}{2}$.
\item $\nu\not=\frac{1}{2}$ and $\frac{2}{q}\leq\frac{n+2}{2}$.
\end{itemize}
\end{theorem}

The proof of Theorem \ref{boundary} rests on the dispersive estimates in the following key lemma.
\begin{lemma}\label{key lemma}
Let $\nu>0$, $n\geq 1$ and $2<q<\infty$. Then we have
\begin{equation}
\|e^{it\mathcal{L}_a^\nu} \mathcal{L}_a^{\frac{2\nu}{q}-\frac{n+2}{2}}f\|_{L_{x,y}^\infty(\mathbb{R}^{2+n})}\lesssim |t|^{-\frac{2}{q}}\|f\|_{L_{x,y}^1(\mathbb{R}^{2+n})},
\end{equation}
if either of the following conditions holds
\begin{itemize}
\item $\nu=\frac{1}{2}$ and $\frac{2}{q}<\frac{n+1}{2}$.
\item $\nu\not=\frac{1}{2}$ and $\frac{2}{q}\leq\frac{n+2}{2}$.
\end{itemize}
\end{lemma}
\begin{proof}
By Theorem \ref{ResultTime}, we have for $0\leq \theta\leq \frac{n+2-n_\nu}{2}$,
\begin{equation*}
\|e^{it\mathcal{L}_a^\nu} \psi(2^{-k}\sqrt{\mathcal{L}_a})f\|_{L_{x,y}^\infty(\mathbb{R}^{2+n})}\lesssim |t|^{-\theta}2^{k(n+2-2\nu\theta)}\|f\|_{L_{x,y}^1(\mathbb{R}^{2+n})},
\end{equation*}
and then
\begin{equation}\label{decay-local}
\|e^{it\mathcal{L}_a^\nu} \psi(2^{-k}\sqrt{\mathcal{L}_a})\mathcal{L}_a^{\frac{2\nu}{q}-\frac{n+2}{2}}f\|_{L_{x,y}^\infty(\mathbb{R}^{2+n})}\lesssim |t|^{-\theta}2^{2k\nu(\frac{2}{q}-\theta)}\|f\|_{L_{x,y}^1(\mathbb{R}^{2+n})}.
\end{equation}
If $\nu=\frac{1}{2}$ and $\frac{2}{q}<\frac{n+1}{2}$, or if $\nu\not=\frac{1}{2}$ and $\frac{2}{q}<\frac{n+2}{2}$, i.e., $\frac{2}{q}<\frac{n+2-n_\nu}{2}$, we get
\begin{align*}
&\quad \|e^{it\mathcal{L}_a^\nu} \mathcal{L}_a^{\frac{2\nu}{q}-\frac{n+2}{2}}f\|_{L_{x,y}^\infty(\mathbb{R}^{2+n})}\\
&\leq \sum_{k\in\mathbb{Z}} \|e^{it\mathcal{L}_a^\nu} \psi(2^{-k}\sqrt{\mathcal{L}_a})\mathcal{L}_a^{\frac{2\nu}{q}-\frac{n+2}{2}}f\|_{L_{x,y}^\infty(\mathbb{R}^{2+n})}\\
&\lesssim \left(\sum_{2^k\leq |t|^{-\frac{1}{2\nu}}}2^{\frac{4k\nu}{q}}+\sum_{2^k\geq|t|^{-\frac{1}{2\nu}}}|t|^{-\frac{n+2-n_\nu}{2}}2^{2k\nu(\frac{2}{q}-\frac{n+2-n_\nu}{2})}\right)\|f\|_{L_{x,y}^1(\mathbb{R}^{2+n})}\\
&\lesssim |t|^{-\frac{2}{q}}\|f\|_{L_{x,y}^1(\mathbb{R}^{2+n})},
\end{align*}
where the second inequality comes from \eqref{decay-local} and the last inequality is obtained by applying Lemma \ref{Sum}. 

Therefore, it remains to prove the case: $\nu\not=\frac{1}{2}$ and $q=\frac{4}{n+2}$. Without loss of generality, we assume $t>0$.
We rewrite
\begin{equation*}
e^{it\mathcal{L}_a^\nu} \mathcal{L}_a^{\frac{2\nu}{q}-\frac{n+2}{2}}f(x,y)=\sum_{k\in\mathbb{Z}}e^{it\mathcal{L}_a^\nu} \psi(2^{-k}t^\frac{1}{2\nu}\sqrt{\mathcal{L}_a})\mathcal{L}_a^{\frac{2\nu}{q}-\frac{n+2}{2}}f(x,y),
\end{equation*}
where 
\begin{align*}
&\quad e^{it\mathcal{L}_a^\nu} \psi(2^{-k}t^\frac{1}{2\nu}\sqrt{\mathcal{L}_a})\mathcal{L}_a^{\frac{2\nu}{q}-\frac{n+2}{2}}f(x,y)\\
&=\int_{\mathbb{R}^{2+n}}\int_0^\infty e^{it\lambda^{2\nu}} \psi(2^{-k}t^\frac{1}{2\nu}\lambda)\lambda^{\frac{4\nu}{q}-(n+2)}dE_{\sqrt{\mathcal{L}_a}}(\lambda; x,y; x',y')f(x',y')dx'dy',
\end{align*}
with the kernel
\begin{equation*}
J_k(x,y;x',y')=\int_0^\infty e^{it\lambda^{2\nu}} \psi(2^{-k}t^\frac{1}{2\nu}\lambda)\lambda^{\frac{4\nu}{q}-(n+2)}dE_{\sqrt{\mathcal{L}_a}}(\lambda; x,y; x',y').
\end{equation*}
By the triangle inequality, it suffices to prove
\begin{equation}\label{key kernel}
\sum_{k\in\mathbb{Z}}|J_k(x,y;x',y')|\lesssim |t|^{-\frac{2}{q}}, \quad\forall (x,y),(x',y')\in\mathbb{R}^{2+n}.
\end{equation}
By \eqref{spectral}, let us divide the kernel $J_k$ into three parts by
\begin{equation*}
J_k(x,y;x',y')=\frac{1}{\pi^{n+1}}\sum_{\pm}\left(J_k^{\pm,1}(x-x',y-y')+J_k^{\pm,2}(x,y;x',y')+J_k^{\pm,3}(x,y;x',y')\right), 
\end{equation*}
where
\begin{equation*}
J_k^{\pm,1}(x,y)=\int_0^\infty  e^{it\lambda^{2\nu}} \psi(2^{-k}t^\frac{1}{2\nu}\lambda)\lambda^{\frac{4\nu}{q}-1}a_{\pm}(\lambda|(x,y)|)e^{\pm i\lambda|(x,y)|}d\lambda,
\end{equation*}
\begin{equation*}
J_k^{\pm,2}(x,y;x',y')=\int_{|\theta-\theta'|}^\pi \int_0^\infty   e^{it\lambda^{2\nu}} \psi(2^{-k}t^\frac{1}{2\nu}\lambda)\lambda^{\frac{4\nu}{q}-1}a_{\pm}(\lambda|\overset{\rightarrow}{n}_s^1|)e^{\pm i\lambda|\overset{\rightarrow}{n}_s^1|}d\lambda A_a(s;\theta,\theta')ds,
\end{equation*}
and
\begin{equation*}
J_k^{\pm,3}(x,y;x',y')=\int_0^\infty \int_0^\infty   e^{it\lambda^{2\nu}} \psi(2^{-k}t^\frac{1}{2\nu}\lambda)\lambda^{\frac{4\nu}{q}-1}a_{\pm}(\lambda|\overset{\rightarrow}{n}_s^2|)e^{\pm i\lambda|\overset{\rightarrow}{n}_s^2|}d\lambda B_a(s;\theta,\theta')ds.
\end{equation*}
Our aim \eqref{key kernel} becomes to prove the following estimates
\begin{align}
\sum_{k\in\mathbb{Z}}|J_k^{\pm,1}(x,y)|&\lesssim |t|^{-\frac{2}{q}},\label{1}\\
\sum_{k\in\mathbb{Z}}|J_k^{\pm,2}(x,y;x',y')|&\lesssim |t|^{-\frac{2}{q}},\label{2}\\
\sum_{k\in\mathbb{Z}}|J_k^{\pm,3}(x,y;x',y')|&\lesssim |t|^{-\frac{2}{q}},\label{3}
\end{align}
for any $(x,y),(x',y')\in\mathbb{R}^{2+n}$. We just prove \eqref{1}. In fact, replacing $|(x,y)|$ by $|\overset{\rightarrow}{n}_s^\ell|, \ell=2,3$ and arguing similarly as \eqref{1}, we can prove \eqref{2} and \eqref{3}. We rewrite $J_k^{\pm,1}$ by
\begin{equation*}
J_k^{\pm,1}(x,y)=|t|^{-\frac{2}{q}}\mathcal{J}_k^{\pm,1}(t^{-\frac{1}{2\nu}}x,t^{-\frac{1}{2\nu}}y)
\end{equation*}
where
\begin{align*}
\mathcal{J}_k^{\pm,1}(x,y)&=2^\frac{4k\nu}{q}\int_0^\infty  e^{i2^{2k\nu}\lambda^{2\nu}} \psi(\lambda)\lambda^{\frac{4\nu}{q}-1}a_{\pm}(2^k\lambda|(x,y)|)e^{\pm i2^k\lambda|(x,y)|}d\lambda\\
&=2^\frac{4k\nu}{q}\int_0^\infty  e^{i2^{2k\nu}\lambda^{2\nu}} \widetilde{\psi}(\lambda)a_{\pm}(2^k\lambda|(x,y)|)e^{\pm i2^k\lambda|(x,y)|}d\lambda,
\end{align*}
with $\widetilde{\psi}(\lambda)=\psi(\lambda)\lambda^{\frac{4\nu}{q}-1}$ such that $\supp \widetilde{\psi}\subseteq [1/2,1]$.

Therefore, the estimate \eqref{1} is equivalent to
\begin{equation}\label{J}
\sum_{k\in\mathbb{Z}}|\mathcal{J}_k^{\pm,1}(x,y)|\lesssim 1,\quad \forall (x,y)\in\mathbb{R}^{2+n}.
\end{equation}
As the proof of part \textbf{(A)} in Theorem \ref{ResultTime}, we shall discuss the above statement in two cases.

\textbf{\underline{Case 1.}} $2^k|(x,y)|\leq 1$. For any $Q\in\mathbb{N}$, using integration by part $Q$-times, we have
\begin{align*}
|\mathcal{J}_k^{\pm,1}(x,y)|&=\left|2^\frac{4k\nu}{q}\int_0^\infty  \left[\left(\frac{\partial_\lambda}{i2\nu 2^{2k\nu}\lambda^{2\nu-1}}\right)^Q e^{i2^{2k\nu}\lambda^{2\nu}}\right] \widetilde{\psi}(\lambda)a_{\pm}(2^k\lambda|(x,y)|)e^{\pm i2^k\lambda|(x,y)|}d\lambda\right|\\
&\lesssim 2^\frac{4k\nu}{q}2^{-2kQ\nu}.
\end{align*}
Hence, choosing $Q$ sufficiently large and applying Lemma \ref{Sum}, it yields
\begin{equation}\label{Case 1}
\begin{aligned}
\sum_{2^k|(x,y)|\leq 1}|\mathcal{J}_k^{\pm,1}(x,y)|&=\sum_{\substack{k\geq 1\\ 2^k|(x,y)|\leq 1}}|\mathcal{J}_k^{\pm,1}(x,y)|+\sum_{\substack{k\leq 0\\ 2^k|(x,y)|\leq 1}}|\mathcal{J}_k^{\pm,1}(x,y)|\\
&\lesssim \sum_{k\geq1}2^\frac{4k\nu}{q}2^{-2kQ\nu}+\sum_{k\leq 0}2^\frac{4k\nu}{q}\\
&\lesssim 1.
\end{aligned}
\end{equation}
\textbf{\underline{Case 2.}} $2^k|(x,y)|>1$. We rewrite $\mathcal{J}_k^{\pm,1}(x,y)$ by
\begin{align*}
\mathcal{J}_k^{\pm,1}(x,y)
&=2^\frac{4k\nu}{q}\int_0^\infty  e^{i\Phi^\pm_k(\lambda;x,y)}\widetilde{\psi}(\lambda)a_{\pm}(2^k\lambda|(x,y)|)d\lambda,
\end{align*} 
with $\Phi^\pm_k(\lambda;x,y)=2^{2k\nu}\lambda^{2\nu} \pm 2^k\lambda|(x,y)|$. We continue our discussion by considering the following three sub-cases:
\begin{itemize}
\item $2^{k(2\nu-1)}\sim |(x,y)|>2^{-k}$. In this case, $2^k|(x,y)|\sim 2^{2k\nu}$ and $|\partial_\lambda^2\Phi^\pm_k(\lambda;x,y)|\sim 2^{2k\nu}$. By the Van der Corput lemma, it indicates that
\begin{equation}
|\mathcal{J}_k^{\pm,1}(x,y)|\lesssim 2^\frac{4k\nu}{q} \left(2^{2k\nu}\right)^{-\frac{1}{2}}\left(2^k|(x,y)|\right)^{-\frac{n+1}{2}}\sim 2^{k\nu(\frac{4}{q}-(n+2))}=1.
\end{equation}
Hence, we have
\begin{equation}\label{Subcase 1}
\sum_{\substack{2^{k(2\nu-1)}\sim |(x,y)|\\ 2^k|(x,y)|>1}}|\mathcal{J}_k^{\pm,1}(x,y)|\lesssim 1.
\end{equation}
\item $|(x,y)|\gg 2^{k(2\nu-1)}$. In this case, $|\partial_\lambda\Phi^\pm_k(\lambda;x,y)|\gtrsim 2^k|(x,y)|\gg 2^{2k\nu}$. For any $Q\in\mathbb{N}$, using integration by part $Q$-times, we have for any $Q\in\mathbb{N}$,
\begin{equation}
\begin{aligned}
|\mathcal{J}_k^{\pm,1}(x,y)|&=\left|2^\frac{4k\nu}{q}\int_0^\infty  \left[\left(\frac{\partial_\lambda}{i\partial_\lambda\Phi^\pm_k(\lambda;x,y)}\right)^Q e^{i\Phi^\pm_k(\lambda;x,y)}\right] \widetilde{\psi}(\lambda)a_{\pm}(2^k\lambda|(x,y)|)d\lambda\right|\\
&\lesssim 2^\frac{4k\nu}{q}\left(2^k|(x,y)|\right)^{-Q-\frac{n+1}{2}}\lesssim 2^\frac{4k\nu}{q}2^{-2k\nu(Q+\frac{n+1}{2})}.
\end{aligned}
\end{equation}
Hence, argued as \eqref{Case 1}, we have
\begin{equation}\label{Subcase 2}
\sum_{\substack{|(x,y)|\gg 2^{k(2\nu-1)}\\ 2^k|(x,y)|>1}}|\mathcal{J}_k^{\pm,1}(x,y)|\lesssim 1.
\end{equation}
\item $|(x,y)|\ll 2^{k(2\nu-1)}$. In this case, $1<2^k|(x,y)|\ll 2^{2k\nu}$ yields $k\geq 0$. Besides, $|\partial_\lambda\Phi^\pm_k(\lambda;x,y)|\gtrsim 2^{2k\nu}$. For any $Q\in\mathbb{N}$, using integration by part $Q$-times, we have for any $Q\in\mathbb{N}$
\begin{equation}
\begin{aligned}
|\mathcal{J}_k^{\pm,1}(x,y)|&=\left|2^\frac{4k\nu}{q}\int_0^\infty  \left[\left(\frac{\partial_\lambda}{i\partial_\lambda\Phi^\pm_k(\lambda;x,y)}\right)^Q e^{i\Phi^\pm_k(\lambda;x,y)}\right] \widetilde{\psi}(\lambda)a_{\pm}(2^k\lambda|(x,y)|)d\lambda\right|\\
&\lesssim 2^\frac{4k\nu}{q}\left(2^{2k\nu}\right)^{-Q}\left(2^k|(x,y)|\right)^{-\frac{n+1}{2}}\lesssim 2^\frac{4k\nu}{q}\left(2^{2k\nu}\right)^{-Q}.
\end{aligned}
\end{equation}
Hence, argued as \eqref{Case 1}, we have
\begin{equation}\label{Subcase 3}
\sum_{\substack{|(x,y)|\ll 2^{k(2\nu-1)}\\ 2^k|(x,y)|>1}}|\mathcal{J}_k^{\pm,1}(x,y)|\lesssim 1.
\end{equation}
\end{itemize}
Therefore, by \eqref{Subcase 1}, \eqref{Subcase 2} and \eqref{Subcase 3}, we obtain
\begin{equation}\label{Case 2}
\begin{aligned}
&\quad\sum_{2^k|(x,y)|>1}|\mathcal{J}_k^{\pm,1}(x,y)|\\
&=\sum_{\substack{2^{k(2\nu-1)}\sim |(x,y)|\\ 2^k|(x,y)|>1}}|\mathcal{J}_k^{\pm,1}(x,y)|+\sum_{\substack{|(x,y)|\gg 2^{k(2\nu-1)}\\ 2^k|(x,y)|>1}}|\mathcal{J}_k^{\pm,1}(x,y)|+\sum_{\substack{|(x,y)|\ll 2^{k(2\nu-1)}\\ 2^k|(x,y)|>1}}|\mathcal{J}_k^{\pm,1}(x,y)|\\
&\lesssim 1.
\end{aligned}
\end{equation}
The estimate \eqref{J} comes immediately from \eqref{Case 1} and \eqref{Case 2}, which completes our proof.\\
\end{proof}

\begin{proof}[Proof of Theorem \ref{boundary}] Note that we can rewrite  \eqref{boundary estimate} by
\begin{equation*}
\|e^{it\mathcal{L}_a^\nu} \mathcal{L}_a^{\frac{2\nu}{q}-\frac{n+2}{4}} f\|_{L_t^q(\mathbb{R},L_{x,y}^\infty(\mathbb{R}^{2+n}))}\lesssim \|f\|_{L_{x,y}^2(\mathbb{R}^{2+n})},
\end{equation*}
By $TT^*$-method, the above estimate is equivalent to
\begin{equation*}
\left\|\int_\mathbb{R} e^{i(t-s)\mathcal{L}_a^\nu}\mathcal{L}_a^{\frac{2\nu}{q}-\frac{n+2}{2}} F(s,\cdot)ds\right\|_{L_t^q(\mathbb{R},L_{x,y}^\infty(\mathbb{R}^{2+n}))}\lesssim \|F\|_{L_t^{q'}(\mathbb{R},L_{x,y}^1(\mathbb{R}^{2+n}))}.
\end{equation*}
Indeed, it follows from the dispersive estimate in Lemma \ref{key lemma} and the Hardy-Littlewood-Sobolev inequality that
\begin{align*}
\left\|\int_\mathbb{R} e^{i(t-s)\mathcal{L}_a^\nu}\mathcal{L}_a^{\frac{2\nu}{q}-\frac{n+2}{2}} F(s,\cdot)ds\right\|_{L_t^q(\mathbb{R},L_{x,y}^\infty(\mathbb{R}^{2+n}))}&\leq \left\|\int_\mathbb{R} \|e^{i(t-s)\mathcal{L}_a^\nu}\mathcal{L}_a^{\frac{2\nu}{q}-\frac{n+2}{2}} F(s,\cdot)\|_{L_{x,y}^\infty(\mathbb{R}^{2+n})}ds\right\|_{L_t^q(\mathbb{R})}\\
&\lesssim \left\|\int_\mathbb{R} |t-s|^{-\frac{2}{q}}\|F(s,\cdot)\|_{L_{x,y}^1(\mathbb{R}^{2+n})}ds\right\|_{L_t^q(\mathbb{R})}\\
&\lesssim \|F\|_{L_t^{q'}(\mathbb{R},L_{x,y}^1(\mathbb{R}^{2+n}))},
\end{align*}
which completes the proof. 
\end{proof}

\section{Applications}\label{sec5}
In this section, we provide several applications of Theorem \ref{ResultTime} by considering some particular smooth function $\phi$. In particular, we prove Strichartz estimates for some concrete wave equations related to the operator $\mathcal{L}_a$ by using Theorem \ref{ResultTime}. We start this section by recalling a variety of the abstract Keel-Tao's Strichartz estimates theorem proved in Zhang \cite{Zhang2019}. This is an analogue of the semiclassical Strichartz estimates for Schr\"odinger in \cite{Koch-Tataru-Zworski, Zworski}.
\begin{proposition}[\cite{Zhang2019}]\label{Keel-Tao argument} Let $(X,\mathcal{M},\nu)$ be a $\sigma$-finite measured space and $U: \mathbb{R}\rightarrow B(L^2(X,\mathcal{M},\mu))$ be a weakly measurable map satisfying, for some constants $C$, $\alpha\in\mathbb{R}$, $\sigma,h>0$,
\begin{align}
\|U(t)\|_{L^2\rightarrow L^2}&\leq C,\quad t\in\mathbb{R},\label{energy estimate}\\
\|U(t)U(s)^*f\|_{L^\infty}&\leq Ch^{-\alpha}(h+|t-s|)^{-\sigma}\|f\|_{L^1}.\label{1 to infty}
\end{align}
Then for every pair $q,r\in[1,\infty]$ such that $(q,r,\sigma)\neq (2,\infty,1)$ and
\begin{equation*}
\frac{1}{q}+\frac{\sigma}{r}\leq\frac{\sigma}{2}, \quad q\geq2,
\end{equation*}
there exists a constant $\widetilde{C}$ only depending on $C,\sigma, q$ and $r$ such that
\begin{equation*}
\left(\int_\mathbb{R} \|U(t)u_0\|^q_{L^r}dt\right)^\frac{1}{q}\leq \widetilde{C}\Lambda(h)\|u_0\|_{L^2},
\end{equation*}
where $\Lambda(h)=h^{-(\alpha+\sigma)(\frac{1}{2}-\frac{1}{r})+\frac{1}{q}}$.
\end{proposition}
We also need to introduce the Littlewood-Paley square function inequality associated with the operator $\mathcal{L}_a$, see \cite{Zhang-Zhang}.
\begin{proposition}[LP square function inequality]\label{LP inequality} For $1<p<\infty$, there exists constants $c_p$ and $C_p$ depending on $p$ such that
\begin{equation*}
c_p\|f\|_{L^p(\mathbb{R}^{2+n})}\leq \left\|\left(\sum_{k\in\mathbb{Z}}|\psi(2^{-k}\sqrt{\mathcal{L}_a})f|^2\right)^\frac{1}{2}\right\|_{L^p(\mathbb{R}^{2+n})}\leq C_p\|f\|_{L^p(\mathbb{R}^{2+n})}.
\end{equation*}
\end{proposition}
 
\subsection{Wave equation}
First, we consider the wave equation as
\begin{equation}\label{Wave-equ}
	\begin{cases}
	\partial_t^2u+\mathcal{L}_a u=0,\;t\in\mathbb{R},\; (x,y)\in\mathbb{R}^{2+n},\\
	u(0,x,y)=f(x,y),\; \partial_tu(0,x,y)=g(x,y).
	\end{cases}
	\end{equation}
By Duhamel's principle, the solution is formally given by
\begin{equation*}
	u(t,x,y)=\frac{e^{it\sqrt{\mathcal{L}_a}}+e^{-it\sqrt{\mathcal{L}_a}}}{2}f(x,y)+\frac{e^{it\sqrt{\mathcal{L}_a}}-e^{-it\sqrt{\mathcal{L}_a}}}{2i\sqrt{\mathcal{L}_a}}g(x,y).
	\end{equation*}
This reduces to the wave semigroup $\{e^{it\sqrt{\mathcal{L}_a}}\}_{t\in\mathbb{R}}$, which corresponds to the case $\phi(\lambda)=\lambda$. Hence, $\phi$ satisfies (C1) and (C2) with $m_1=m_2=1$.  Using Theorem \ref{ResultTime}, we obtain the  following Strichartz estimate for the wave operator $e^{it\sqrt{\mathcal{L}_a}}$. 
\begin{theorem}\label{Intermediate-wave}  Let $n\geq1$, $s\geq0$ and $q,r\in [2,\infty]$ satisfy
\begin{equation*}
\frac{2}{q}+\frac{n+1}{r}\leq \frac{n+1}{2},\quad s=(n+2)\left(\frac{1}{2}-\frac{1}{r}\right)-\frac{1}{q}, \quad r<\infty
\end{equation*}
or $2<q<\infty$, $r=\infty$, $s=\frac{n+2}{2}-\frac{1}{q}$. Then we have
\begin{equation*}
\|e^{it\sqrt{\mathcal{L}_a}}f\|_{L_t^q(\mathbb{R},L^r_{x,y}(\mathbb{R}^{n+2}))}\lesssim \|f\|_{\dot{H}^s_a(\mathbb{R}^{n+2})}.
\end{equation*}
\end{theorem}
\begin{proof} Let $U_k(t)=e^{it\sqrt{\mathcal{L}_a}}\psi(2^{-k}\sqrt{\mathcal{L}_a})$, for any $k\in\mathbb{Z}$.
It follows from Theorem \ref{ResultTime} that
\begin{equation*}
\|U_k(t)f\|_{L^\infty_{x,y}(\mathbb{R}^{n+2})}\lesssim |t|^{-\theta}2^{k(n+2-\theta)}\|f\|_{L^1_{x,y}(\mathbb{R}^{n+2})},\quad \forall 0\leq \theta\leq \frac{n+1}{2},\; \forall k\in\mathbb{Z},
\end{equation*}
which implies that
\begin{equation*}
\|U_k(t)f\|_{L^\infty_{x,y}(\mathbb{R}^{n+2})}\lesssim 2^{\frac{n+3}{2}k}(2^{-k}+|t|)^{-\frac{n+1}{2}}\|f\|_{L^1_{x,y}(\mathbb{R}^{n+2})},\quad \forall k\in\mathbb{Z}.
\end{equation*}
Then we have
\begin{equation*}
\|U_k(t)U_k(s)^*\|_{L^1_{x,y}(\mathbb{R}^{n+2})\rightarrow L^\infty_{x,y}(\mathbb{R}^{n+2})}\lesssim 2^{\frac{n+3}{2}k}(2^{-k}+|t-s|)^{-\frac{n+1}{2}},\quad \forall k\in\mathbb{Z},
\end{equation*}
which satisfies the estimates \eqref{1 to infty} in Proposition \ref{Keel-Tao argument} for $\alpha=\frac{n+3}{2}$, $\sigma=\frac{n+1}{2}$ and $h=2^{-k}$. 
On the other hand, by Plancherel's formula {\color{red} \eqref{Plancherel}}, we get the energy estimate
\begin{equation*}
\|U_k(t)\|_{L^2_{x,y}(\mathbb{R}^{n+2})\rightarrow L^2_{x,y}(\mathbb{R}^{n+2})}\lesssim 1,\quad \forall k\in\mathbb{Z}.
\end{equation*}
Using Keel-Tao's argument in Proposition \ref{Keel-Tao argument}, it gives that
\begin{equation*}
\|U_k(t)f\|_{L_t^q(\mathbb{R},L^r_{x,y}(\mathbb{R}^{2+n}))}\lesssim 2^{k[(n+2)(\frac{1}{2}-\frac{1}{r})-\frac{1}{q}]}\|f\|_{L^2(\mathbb{R}^{2+n})}.
\end{equation*}
for $q\in[2,\infty]$, $r\in [2,\infty)$ such that $\frac{2}{q}+\frac{n+1}{r}\leq \frac{n+1}{2}$. Notice that
\begin{equation*}
e^{it\sqrt{\mathcal{L}_a}}f=\sum_{k\in\mathbb{Z}}U_k(t)f=\sum_{j\in\mathbb{Z}}\sum_{k\in\mathbb{Z}}U_k(t)\psi(2^{-j}\sqrt{\mathcal{L}_a})f,
\end{equation*}
and $U_k(t)\psi(2^{-j}\sqrt{\mathcal{L}_a})f$ vanishes if $|j-k|\geq3$. Hence, using the Littlewood-Paley square inequality in Proposition \ref{LP inequality} and the Minkowski inequality, we have
\begin{align*}
\|e^{it\sqrt{\mathcal{L}_a}}f\|_{L_t^q(\mathbb{R},L^r_{x,y}(\mathbb{R}^{n+2}))}&\lesssim \left\|\left(\sum_{j\in\mathbb{Z}}|\sum_{k\in\mathbb{Z}}U_k(t)\psi(2^{-j}\sqrt{\mathcal{L}_a})f|^2\right)^\frac{1}{2}\right\|_{L_t^q(\mathbb{R},L^r_{x,y}(\mathbb{R}^{n+2}))}\\
&\lesssim \left\|\left(\sum_{j\in\mathbb{Z}}\sum_{|k-j|\leq 2}|U_k(t)\psi(2^{-j}\sqrt{\mathcal{L}_a})f|^2\right)^\frac{1}{2}\right\|_{L_t^q(\mathbb{R},L^r_{x,y}(\mathbb{R}^{n+2}))}\\
&\lesssim \left(\sum_{j\in\mathbb{Z}}\sum_{|k-j|\leq 2}\left\|U_k(t)\psi(2^{-j}\sqrt{\mathcal{L}_a})f\right\|_{L_t^q(\mathbb{R},L^r_{x,y}(\mathbb{R}^{n+2}))}^2\right)^\frac{1}{2}\\
&\lesssim \left(\sum_{j\in\mathbb{Z}}\sum_{|k-j|\leq 2}2^{2k[(n+2)(\frac{1}{2}-\frac{1}{r})-\frac{1}{q}]}\|\psi(2^{-j}\sqrt{\mathcal{L}_a})f\|^2_{L^2(\mathbb{R}^{2+n})}\right)^\frac{1}{2}\\
&\lesssim \left(\sum_{j\in\mathbb{Z}}2^{2j[(n+2)(\frac{1}{2}-\frac{1}{r})-\frac{1}{q}]}\|\psi(2^{-j}\sqrt{\mathcal{L}_a})f\|^2_{L^2(\mathbb{R}^{2+n})}\right)^\frac{1}{2}\\
&=\|f\|_{\dot{H}_a^{(n+2)(\frac{1}{2}-\frac{1}{r})-\frac{1}{q}}}.
\end{align*}
If $2<q<\infty$ and $r=\infty$, it follows directly from Theorem \ref{boundary} when $\nu=\frac{1}{2}$ that
\begin{equation*}
\|e^{it\sqrt{\mathcal{L}_a}}f\|_{L_t^q(\mathbb{R},L^r_{x,y}(\mathbb{R}^{n+2}))}\lesssim \|f\|_{\dot{H}^{\frac{n+2}{2}-\frac{1}{q}}_a(\mathbb{R}^{n+2})}.
\end{equation*}
\end{proof}
It immediately yields the Strichartz inequality for the solution of the wave equation \eqref{Wave-equ}. 
\begin{corollary}
Under the same hypotheses as in Theorem \ref{Intermediate-wave}, the solution $u$ of  the wave equation \eqref{Wave-equ} satisfies the following estimate
	\begin{equation*}
	\|u\|_{L_t^q(\mathbb{R},L^r_{x,y}(\mathbb{R}^{2+n}))}\lesssim \|f\|_{\dot{H}^s_a(\mathbb{R}^{2+n})}+\|g\|_{\dot{H}^{s-1}_a(\mathbb{R}^{2+n})}.
	\end{equation*}
\end{corollary}
\begin{remark} In addition to proving Theorem \ref{wave-Strichartz} established by Zhang-Zhang \cite{Zhang-Zhang}, we further derive boundary Strichartz estimates for the wave equation \eqref{Wave-equ}.
\end{remark} 
\subsection{Fractional Schr\"{o}dinger equation} Consider the fractional Schr\"{o}dinger equation ($0<\nu\neq \frac{1}{2}$)
    \begin{equation}\label{FSchrEqu}
	\begin{cases}
	i\partial_tu+\mathcal{L}_a^\nu u=0,\;t\in\mathbb{R},\; (x,y)\in\mathbb{R}^{2+n},\\
	u(0,x,y)=f(x,y).
	\end{cases}
	\end{equation}
By Duhamel's principle, the solution is formally given by
	\begin{equation*}\label{solution1}
	u(t,x,y)=e^{it\mathcal{L}_a^\nu}f(x,y),
	\end{equation*}
and it corresponds to the case when $\phi(\lambda)=\lambda^{2\nu}$. By a simple calculation, we see that 
	\begin{equation*}
	\phi'(\lambda)=2\nu \lambda^{2\nu-1}, \quad
	\phi''(\lambda)=2\nu(2\nu-1)\lambda^{2\nu-2}.
	\end{equation*}
This shows that $\phi$ satisfies (C1)-(C4) with $m_1=\alpha_1=m_2=\alpha_2=2\nu$. Using Theorem \ref{ResultTime}, we obtain the following Strichartz estimate for the fractional Schr\"odinger operator $e^{it\mathcal{L}_a^\nu}$.
\begin{theorem}\label{Intermediate-fractional} For any $0<\nu\neq\frac{1}{2}$, let $s\in\mathbb{R}$ and $q,r\in [2,\infty]$ satisfy
\begin{equation*}
\frac{2}{q}+\frac{n+2}{r}\leq \frac{n+2}{2},\quad s=(n+2)\left(\frac{1}{2}-\frac{1}{r}\right)-\frac{2\nu}{q}, \quad r<\infty
\end{equation*}
or $2<q<\infty$, $r=\infty$, $s=\frac{n+2}{2}-\frac{2\nu}{q}$. Then we have
\begin{equation*}
\|e^{it\mathcal{L}_a^\nu}f\|_{L_t^q(\mathbb{R},L^r_{x,y}(\mathbb{R}^{n+2}))}\lesssim \|f\|_{\dot{H}^s_a(\mathbb{R}^{n+2})}.
\end{equation*}
\end{theorem}
\begin{proof} Denote $U_k(t)=e^{it\mathcal{L}_a^\nu}\psi(2^{-k}\sqrt{\mathcal{L}_a})$, for any $k\in\mathbb{Z}$.
It follows from Theorem \ref{ResultTime} that
\begin{equation*}
\|U_k(t)f\|_{L^\infty_{x,y}(\mathbb{R}^{n+2})}\lesssim |t|^{-\theta}2^{k(n+2-2\nu\theta)}\|f\|_{L^1_{x,y}(\mathbb{R}^{n+2})},\quad \forall 0\leq \theta\leq \frac{n+2}{2},\; \forall k\in\mathbb{Z},
\end{equation*}
which implies that
\begin{equation*}
\|U_k(t)f\|_{L^\infty_{x,y}(\mathbb{R}^{n+2})}\lesssim 2^{(n+2)(1-\nu)k}(2^{-2k\nu}+|t|)^{-\frac{n+2}{2}}\|f\|_{L^1_{x,y}(\mathbb{R}^{n+2})},\quad \forall k\in\mathbb{Z}.
\end{equation*}
Then we have
\begin{equation*}
\|U_k(t)U_k(s)^*\|_{L^1_{x,y}(\mathbb{R}^{n+2})\rightarrow L^\infty_{x,y}(\mathbb{R}^{n+2})}\lesssim 2^{(n+2)(1-\nu)k}(2^{-2k\nu}+|t-s|)^{-\frac{n+2}{2}},\quad \forall k\in\mathbb{Z},
\end{equation*}
which satisfies the estimates \eqref{1 to infty} in Proposition \ref{Keel-Tao argument} for $\alpha=\frac{(n+2)(1-\nu)}{2\nu}$, $\sigma=\frac{n+2}{2}$ and $h=2^{-2k\nu}$. 
On the other hand, by Plancherel's formula \eqref{Plancherel}, we get the energy estimate
\begin{equation*}
\|U_k(t)\|_{L^2_{x,y}(\mathbb{R}^{n+2})\rightarrow L^2_{x,y}(\mathbb{R}^{n+2})}\lesssim 1,\quad \forall k\in\mathbb{Z}.
\end{equation*}
Using Keel-Tao's argument in Proposition \ref{Keel-Tao argument}, it gives that
\begin{equation*}
\|U_k(t)f\|_{L_t^q(\mathbb{R},L^r_{x,y}(\mathbb{R}^{2+n}))}\lesssim 2^{k[(n+2)(\frac{1}{2}-\frac{1}{r})-\frac{2\nu}{q}]}\|f\|_{L^2(\mathbb{R}^{2+n})}.
\end{equation*}
for $q\in[2,\infty]$, $r\in [2,\infty)$ such that $\frac{2}{q}+\frac{n+2}{r}\leq \frac{n+2}{2}$. Notice that
\begin{equation*}
e^{it{\mathcal{L}_a^{\nu}}}f=\sum_{k\in\mathbb{Z}}U_k(t)f=\sum_{j\in\mathbb{Z}}\sum_{k\in\mathbb{Z}}U_k(t)\psi(2^{-j}\sqrt{\mathcal{L}_a})f,
\end{equation*}
and $U_k(t)\psi(2^{-j}\sqrt{\mathcal{L}_a})f$ vanishes if $|j-k|\geq3$. Hence, using the Littlewood-Paley square inequality in Proposition \ref{LP inequality} and the Minkowski inequality, we have
\begin{align*}
\|e^{it\mathcal{L}_a^\nu}f\|_{L_t^q(\mathbb{R},L^r_{x,y}(\mathbb{R}^{n+2}))}&\lesssim \left\|\left(\sum_{j\in\mathbb{Z}}|\sum_{k\in\mathbb{Z}}U_k(t)\psi(2^{-j}\sqrt{\mathcal{L}_a})f|^2\right)^\frac{1}{2}\right\|_{L_t^q(\mathbb{R},L^r_{x,y}(\mathbb{R}^{n+2}))}\\
&\lesssim \left\|\left(\sum_{j\in\mathbb{Z}}\sum_{|k-j|\leq 2}|U_k(t)\psi(2^{-j}\sqrt{\mathcal{L}_a})f|^2\right)^\frac{1}{2}\right\|_{L_t^q(\mathbb{R},L^r_{x,y}(\mathbb{R}^{n+2}))}\\
&\lesssim \left(\sum_{j\in\mathbb{Z}}\sum_{|k-j|\leq 2}\left\|U_k(t)\psi(2^{-j}\sqrt{\mathcal{L}_a})f\right\|_{L_t^q(\mathbb{R},L^r_{x,y}(\mathbb{R}^{n+2}))}^2\right)^\frac{1}{2}\\
&\lesssim \left(\sum_{j\in\mathbb{Z}}\sum_{|k-j|\leq 2}2^{2k[(n+2)(\frac{1}{2}-\frac{1}{r})-\frac{2\nu}{q}]}\|\psi(2^{-j}\sqrt{\mathcal{L}_a})f\|^2_{L^2(\mathbb{R}^{2+n})}\right)^\frac{1}{2}\\
&\lesssim \left(\sum_{j\in\mathbb{Z}}2^{2j[(n+2)(\frac{1}{2}-\frac{1}{r})-\frac{2\nu}{q}]}\|\psi(2^{-j}\sqrt{\mathcal{L}_a})f\|^2_{L^2(\mathbb{R}^{2+n})}\right)^\frac{1}{2}\\
&=\|f\|_{\dot{H}_a^{(n+2)(\frac{1}{2}-\frac{1}{r})-\frac{2\nu}{q}}}.
\end{align*}
If $2<q<\infty$ and $r=\infty$, it follows directly from Theorem \ref{boundary} that
\begin{equation*}
\|e^{it\mathcal{L}_a^\nu}f\|_{L_t^q(\mathbb{R},L^r_{x,y}(\mathbb{R}^{n+2}))}\lesssim \|f\|_{\dot{H}^{\frac{n+2}{2}-\frac{2\nu}{q}}_a(\mathbb{R}^{n+2})}.
\end{equation*}
\end{proof}
In particular, when $\nu=1$, the fractional Schr\"{o}dinger equation \eqref{FSchrEqu} reduces to the Schr\"odinger equation 
\begin{equation}\label{SchrEqu}
	\begin{cases}
	i\partial_tu+\mathcal{L}_au=0,\;t\in\mathbb{R},\; (x,y)\in\mathbb{R}^{2+n},\\
	u(0,x,y)=f(x,y).
	\end{cases}
	\end{equation}
As a corollary, we obtain Strichartz estimates for the solution to the Schr\"odinger equation \eqref{SchrEqu}.
\begin{corollary} Let $s\in\mathbb{R}$ and $q,r\in [2,\infty]$ satisfy
\begin{equation*}
\frac{2}{q}+\frac{n+2}{r}\leq \frac{n+2}{2},\quad s=(n+2)\left(\frac{1}{2}-\frac{1}{r}\right)-\frac{2}{q}, \quad r<\infty
\end{equation*}
or $2<q<\infty$, $r=\infty$, $s=\frac{n+2}{2}-\frac{2}{q}$. Then the solution of the Schr\"odinger equation \eqref{SchrEqu} satisfies
\begin{equation*}
\|u\|_{L_t^q(\mathbb{R},L^r_{x,y}(\mathbb{R}^{n+2}))}=\|e^{it\mathcal{L}_a}f\|_{L_t^q(\mathbb{R},L^r_{x,y}(\mathbb{R}^{n+2}))}\lesssim \|f\|_{\dot{H}^s_a(\mathbb{R}^{n+2})}.
\end{equation*}
\end{corollary}
\begin{remark} When $s=0$ and $r<\infty$, this reduces to Theorem \ref{Schrodinger-Strichartz} proved by Zhang-Zhang \cite{Zhang-Zhang}.
\end{remark}
\subsection{Fourth-order Schr\"{o}dinger equation}
Consider the fourth-order Schr\"{o}dinger equation
	\begin{equation}\label{SSchrEqu}
	\begin{cases}
	i\partial_tu+\mathcal{L}_a^2u+\mathcal{L}_a u=0,\;t\in\mathbb{R},\; (x,y)\in\mathbb{R}^{2+n},\\
	u(0,x,y)=f(x,y).
	\end{cases}
	\end{equation}
By Duhamel's principle, the solution is formally given by
	\begin{equation*}\label{solution3}
	u(t,x,y)=e^{it(\mathcal{L}_a^2+\mathcal{L}_a)}f(x,y),
	\end{equation*}
which corresponds to the case when $\phi(\lambda)=\lambda^4+\lambda^2$. By a simple calculation, we see that 
	\begin{equation*}
	\phi'(r)=4\lambda^3+2\lambda, \quad
	\phi''(r)=12\lambda^2+2.
	\end{equation*}
Thus, $\phi$ satisfies (C1)-(C4) with $m_1=\alpha_1=4$, $m_2=\alpha_2=2$.  Using  Theorem \ref{ResultTime}, we obtain the following Strichartz estimate for  $e^{it(\mathcal{L}_a^2+\mathcal{L}_a)}$. 
\begin{theorem}
Let $s\in\mathbb{R}$, $q\in [2,\infty]$ and  $r\in [2,\infty)$ satisfy
\begin{equation*}
\frac{2}{q}+\frac{n+2}{r}\leq \frac{n+2}{2},\quad s=(n+2)\left(\frac{1}{2}-\frac{1}{r}\right)-\frac{4}{q}.
\end{equation*}
Then the solution of the Schr\"odinger equation \eqref{SSchrEqu} satisfies
\begin{equation*}
\|u\|_{L_t^q(\mathbb{R},L^r_{x,y}(\mathbb{R}^{n+2}))}=\|e^{it(\mathcal{L}_a^2+\mathcal{L}_a)}f\|_{L_t^q(\mathbb{R},L^r_{x,y}(\mathbb{R}^{n+2}))}\lesssim \|f\|_{H^s_a(\mathbb{R}^{n+2})}.
\end{equation*}
\end{theorem}
\begin{proof} Denote $U^{low}(t)=e^{it(\mathcal{L}_a^2+\mathcal{L}_a)}\varphi(\sqrt{\mathcal{L}_a})$ and $U_k(t)=e^{it(\mathcal{L}_a^2+\mathcal{L}_a)}\psi(2^{-k}\sqrt{\mathcal{L}_a})$, for any $k\in\mathbb{Z}$.
It follows from Theorem \ref{ResultTime} that for any $0\leq \theta\leq \frac{n+2}{2}$
\begin{align*}
\|U_k(t)f\|_{L^\infty_{x,y}(\mathbb{R}^{n+2})}&\lesssim |t|^{-\theta}2^{k(n+2-4\theta)}\|f\|_{L^1_{x,y}(\mathbb{R}^{n+2})},\quad  \forall k\geq1,\\
\|U^{low}(t)f\|_{L^\infty_{x,y}(\mathbb{R}^{n+2})}&\lesssim (1+|t|)^{-\theta}\|f\|_{L^1_{x,y}(\mathbb{R}^{n+2})},
\end{align*}
which implies that
\begin{align*}
\|U_k(t)f\|_{L^\infty_{x,y}(\mathbb{R}^{n+2})}&\lesssim 2^{-(n+2)k}(2^{-4k}+|t|)^{-\frac{n+2}{2}}\|f\|_{L^1_{x,y}(\mathbb{R}^{n+2})},\quad \forall k\geq1,\\
\|U^{low}(t)f\|_{L^\infty_{x,y}(\mathbb{R}^{n+2})}&\lesssim (1+|t|)^{-\frac{n+2}{2}}\|f\|_{L^1_{x,y}(\mathbb{R}^{n+2})},
\end{align*}
Then we have
\begin{align*}
\|U_k(t)U_k(s)^*\|_{L^1_{x,y}(\mathbb{R}^{n+2})\rightarrow L^\infty_{x,y}(\mathbb{R}^{n+2})}&\lesssim 2^{-(n+2)k}(2^{-4k}+|t-s|)^{-\frac{n+2}{2}},\quad \forall k\geq1,\\
\|U^{low}(t)U^{low}(s)^*\|_{L^1_{x,y}(\mathbb{R}^{n+2})\rightarrow L^\infty_{x,y}(\mathbb{R}^{n+2})}&\lesssim (1+|t-s|)^{-\frac{n+2}{2}},
\end{align*}
where $U_k(t),k\geq1$ satisfies the estimates \eqref{1 to infty} in Proposition \ref{Keel-Tao argument} for $\alpha=-\frac{n+2}{4}$, $\sigma=\frac{n+2}{2}$ and $h=2^{-4k}$. 
On the other hand, by Plancherel's formula {\color{red} \eqref{Plancherel}}, we get the energy estimate
\begin{equation*}
\|U^{low}(t)\|_{L^2_{x,y}(\mathbb{R}^{n+2})\rightarrow L^2_{x,y}(\mathbb{R}^{n+2})}\lesssim 1,\quad \|U_k(t)\|_{L^2_{x,y}(\mathbb{R}^{n+2})\rightarrow L^2_{x,y}(\mathbb{R}^{n+2})}\lesssim 1,\quad \forall k\geq1.
\end{equation*}
Using Keel-Tao's argument in Proposition \ref{Keel-Tao argument} and Theorem 1.2 in Keel-Tao \cite{KT}, it gives that
\begin{align*}
\|U_k(t)f\|_{L_t^q(\mathbb{R},L^r_{x,y}(\mathbb{R}^{2+n}))}&\lesssim 2^{k[(n+2)(\frac{1}{2}-\frac{1}{r})-\frac{4}{q}]}\|f\|_{L^2(\mathbb{R}^{2+n})}, \quad\forall k\geq1,\\
\|U^{low}(t)f\|_{L_t^q(\mathbb{R},L^r_{x,y}(\mathbb{R}^{2+n}))}&\lesssim \|f\|_{L^2(\mathbb{R}^{2+n})},
\end{align*}
for $q\in[2,\infty]$, $r\in [2,\infty)$ such that $\frac{2}{q}+\frac{n+2}{r}\leq \frac{n+2}{2}$. Notice that
\begin{equation*}
e^{it(\mathcal{L}_a^2+\mathcal{L}_a)}f=U^{low}(t)f+\sum_{k=1}^\infty U_k(t)f,
\end{equation*}
and $U_k(t)\psi(2^{-j}\sqrt{\mathcal{L}_a})f$ vanishes if $|j-k|\geq3$. Hence, 
\begin{align*}
\|U^{low}(t)f\|_{L_t^q(\mathbb{R},L^r_{x,y}(\mathbb{R}^{n+2}))}&=\|U^{low}(t)(\varphi(\sqrt{\mathcal{L}_a})+\sum_{j=1}^2\psi(2^{-j}\sqrt{\mathcal{L}_a}))f\|_{L_t^q(\mathbb{R},L^r_{x,y}(\mathbb{R}^{n+2}))}\\
&\lesssim \|(\varphi(\sqrt{\mathcal{L}_a})+\sum_{j=1}^2\psi(2^{-j}\sqrt{\mathcal{L}_a}))f\|_{L^2_{x,y}(\mathbb{R}^{n+2})}\\
&\leq \|\varphi(\sqrt{\mathcal{L}_a})f\|_{L^2_{x,y}(\mathbb{R}^{n+2})}+\sum_{j=1}^2\|\psi(2^{-j}\sqrt{\mathcal{L}_a})f\|_{L^2_{x,y}(\mathbb{R}^{n+2})},
\end{align*}
and
\begin{align*}
&\quad \|\sum_{k=1}^\infty U_k(t)f\|_{L_t^q(\mathbb{R},L^r_{x,y}(\mathbb{R}^{n+2}))}\\
&=\|\sum_{j\in\mathbb{Z}}\sum_{k=1}^\infty U_k(t)\psi(2^{-j}\sqrt{\mathcal{L}_a})f\|_{L_t^q(\mathbb{R},L^r_{x,y}(\mathbb{R}^{n+2}))}\\
&=\Big\|\sum_{j=-1}^3U_1(t)\psi(2^{-j}\sqrt{\mathcal{L}_a})f+\sum_{j=0}^4U_2(t)\psi(2^{-j}\sqrt{\mathcal{L}_a})f\\
&\qquad\qquad\qquad\qquad\qquad\qquad\qquad+\sum_{j\in\mathbb{Z}}^\infty\sum_{k=3}^\infty U_k(t)\psi(2^{-j}\sqrt{\mathcal{L}_a})f\Big\|_{L_t^q(\mathbb{R},L^r_{x,y}(\mathbb{R}^{n+2}))}.\\
&\leq \sum_{j=-1}^3\|U_1(t)\psi(2^{-j}\sqrt{\mathcal{L}_a})f\|_{L_t^q(\mathbb{R},L^r_{x,y}(\mathbb{R}^{n+2}))}+\sum_{j=0}^4\|U_2(t)\psi(2^{-j}\sqrt{\mathcal{L}_a})f\|_{L_t^q(\mathbb{R},L^r_{x,y}(\mathbb{R}^{n+2}))}\\
&\qquad\qquad\qquad\qquad\qquad\qquad\qquad+\|\sum_{j\in\mathbb{Z}}^\infty\sum_{k=3}^\infty U_k(t)\psi(2^{-j}\sqrt{\mathcal{L}_a})f\|_{L_t^q(\mathbb{R},L^r_{x,y}(\mathbb{R}^{n+2}))}.\\
&\lesssim \sum_{j=-1}^4\|\psi(2^{-j}\sqrt{\mathcal{L}_a})f\|_{L^2_{x,y}(\mathbb{R}^{n+2})}+\|\sum_{j\in\mathbb{Z}}^\infty\sum_{k=3}^\infty U_k(t)\psi(2^{-j}\sqrt{\mathcal{L}_a})f\|_{L_t^q(\mathbb{R},L^r_{x,y}(\mathbb{R}^{n+2}))}.
\end{align*}
Using the Littlewood-Paley square inequality in Proposition \ref{LP inequality} and the Minkowski inequality, we have
\begin{align*}
&\quad\|\sum_{j\in\mathbb{Z}}^\infty\sum_{k=3}^\infty U_k(t)\psi(2^{-j}\sqrt{\mathcal{L}_a})f\|_{L_t^q(\mathbb{R},L^r_{x,y}(\mathbb{R}^{n+2}))}\\
&\lesssim \|\Big(\sum_{j\in\mathbb{Z}}^\infty |\sum_{k=3}^\infty U_k(t)\psi(2^{-j}\sqrt{\mathcal{L}_a})f|^2\Big)^\frac{1}{2}\|_{L_t^q(\mathbb{R},L^r_{x,y}(\mathbb{R}^{n+2}))}\\
&\lesssim \|\Big(\sum_{j\in\mathbb{Z}}^\infty \sum_{\substack{k\geq 3\\ |k-j|\leq 2}} |U_k(t)\psi(2^{-j}\sqrt{\mathcal{L}_a})f|^2\Big)^\frac{1}{2}\|_{L_t^q(\mathbb{R},L^r_{x,y}(\mathbb{R}^{n+2}))}\\
&\leq \Big(\sum_{j\in\mathbb{Z}}^\infty \sum_{\substack{k\geq 3\\ |k-j|\leq 2}} \|U_k(t)\psi(2^{-j}\sqrt{\mathcal{L}_a})f\|_{L_t^q(\mathbb{R},L^r_{x,y}(\mathbb{R}^{n+2}))}^2\Big)^\frac{1}{2}\\
&\lesssim \|\Big(\sum_{j\in\mathbb{Z}}^\infty \sum_{\substack{k\geq 3\\ |k-j|\leq 2}} 2^{2k[(n+2)(\frac{1}{2}-\frac{1}{r})-\frac{4}{q}]}\|\psi(2^{-j}\sqrt{\mathcal{L}_a})f\|^2_{L^2(\mathbb{R}^{2+n})}\Big)^\frac{1}{2}\|_{L_t^q(\mathbb{R},L^r_{x,y}(\mathbb{R}^{n+2}))}\\
&\lesssim \Big(\sum_{j\geq1}^\infty 2^{2j[(n+2)(\frac{1}{2}-\frac{1}{r})-\frac{4}{q}]}\|\psi(2^{-j}\sqrt{\mathcal{L}_a})f\|^2_{L^2(\mathbb{R}^{2+n})}\Big)^\frac{1}{2}.
\end{align*}
Therefore, we get
\begin{align*}
&\quad \|e^{it(\mathcal{L}_a^2+\mathcal{L}_a)}f\|_{L_t^q(\mathbb{R},L^r_{x,y}(\mathbb{R}^{n+2}))}\\
&\lesssim \|U^{low}(t)f\|_{L_t^q(\mathbb{R},L^r_{x,y}(\mathbb{R}^{n+2}))}+\|\sum_{k=1}^\infty U_k(t)f\|_{L_t^q(\mathbb{R},L^r_{x,y}(\mathbb{R}^{n+2}))}\\
&\lesssim\|\varphi(\sqrt{\mathcal{L}_a})f\|_{L^2_{x,y}(\mathbb{R}^{n+2})}+\sum_{j=-1}^4\|\psi(2^{-j}\sqrt{\mathcal{L}_a})f\|_{L^2_{x,y}(\mathbb{R}^{n+2})}\\
&\qquad\qquad\qquad\qquad+\Big(\sum_{j\geq1}^\infty 2^{2j[(n+2)(\frac{1}{2}-\frac{1}{r})-\frac{4}{q}]}\|\psi(2^{-j}\sqrt{\mathcal{L}_a})f\|^2_{L^2(\mathbb{R}^{2+n})}\Big)^\frac{1}{2}\\
&\lesssim\|\varphi(\sqrt{\mathcal{L}_a})f\|_{L^2_{x,y}(\mathbb{R}^{n+2})}+\Big(\sum_{j\geq1}^\infty 2^{2j[(n+2)(\frac{1}{2}-\frac{1}{r})-\frac{4}{q}]}\|\psi(2^{-j}\sqrt{\mathcal{L}_a})f\|^2_{L^2(\mathbb{R}^{2+n})}\Big)^\frac{1}{2}\\
&=\|f\|_{H_a^{(n+2)(\frac{1}{2}-\frac{1}{r})-\frac{4}{q}}},
\end{align*}
which completes the proof.
\end{proof}

\subsection{Klein-Gordon equation}
Moreover, we consider the Klein-Gordon equation
\begin{equation}\label{K-GEqu}
	\begin{cases}
	\partial_t^2u+\mathcal{L}_a u+u=0,\;t\in\mathbb{R},\; (x,y)\in\mathbb{R}^{2+n},\\
	u(0,x,y)=f(x,y),\; \partial_tu(0,x,y)=g(x,y).
	\end{cases}
	\end{equation}
By Duhamel's principle, the solution is formally given by
\begin{equation*}
	u(t,x,y)=\frac{e^{it\sqrt{1+\mathcal{L}_a}}+e^{-it\sqrt{1+\mathcal{L}_a}}}{2}f(x,y)+\frac{e^{it\sqrt{1+\mathcal{L}_a}}-e^{-it\sqrt{1+\mathcal{L}_a}}}{2i\sqrt{1+\mathcal{L}_a}}g(x,y).
	\end{equation*}
So we naturally consider the operator $e^{it\sqrt{1+\mathcal{L}_a}}$, which corresponds to the case when $\phi(\lambda)=\sqrt{1+\lambda^2}$.
By a simple calculation have
	\begin{equation*}
	\phi'(\lambda)=\lambda(1+\lambda^2)^{-\frac{1}{2}}, \;
	\phi''(\lambda)=(1+\lambda^2)^{-\frac{3}{2}}.
	\end{equation*}
This shows that $\phi$ satisfies (C1)-(C4) with $m_1=1,\alpha_1=-1, $ and $m_2=\alpha_2=2$. Using Theorem \ref{ResultTime}, we obtain the following Strichartz estimate for $e^{it\sqrt{1+\mathcal{L}_a}}$.
\begin{theorem}\label{Klein-Gordon estimate}
For any $\eta\in [0,1]$. Let $s\geq0$, $q\in [2,\infty]$ and $r\in [2,\infty)$satisfy
\begin{equation*}
\frac{2}{q}+\frac{n+1+\eta}{r}\leq \frac{n+1+\eta}{2},\quad s=(n+2+\eta)\left(\frac{1}{2}-\frac{1}{r}\right)-\frac{1}{q}.
\end{equation*}
Then we have
\begin{equation*}
\|e^{it\sqrt{1+\mathcal{L}_a}}f\|_{L_t^q(\mathbb{R},L^r_{x,y}(\mathbb{R}^{n+2}))}\lesssim \|f\|_{H^s_a(\mathbb{R}^{n+2})}.
\end{equation*}
\end{theorem}
\begin{proof} Denote $U^{low}(t)=e^{it\sqrt{1+\mathcal{L}_a}}\varphi(\sqrt{\mathcal{L}_a})$ and $U_k(t)=e^{it\sqrt{1+\mathcal{L}_a}}\psi(2^{-k}\sqrt{\mathcal{L}_a})$, for any $k\in\mathbb{Z}$.
It follows from Theorem \ref{ResultTime} that for any $0\leq \theta\leq \frac{n+1}{2}$, $k\geq1$ and $0\leq \eta\leq 1$,
\begin{align*}
\|U_k(t)f\|_{L^\infty_{x,y}(\mathbb{R}^{n+2})}&\lesssim |t|^{-\theta}2^{k(n+2-\theta)}\|f\|_{L^1_{x,y}(\mathbb{R}^{n+2})},\\
\|U_k(t)f\|_{L^\infty_{x,y}(\mathbb{R}^{n+2})}&\lesssim |t|^{-\frac{n+1+\eta}{2}}2^{k(n+2-\frac{n+1-\eta}{2})}\|f\|_{L^1_{x,y}(\mathbb{R}^{n+2})},\\
\|U^{low}(t)f\|_{L^\infty_{x,y}(\mathbb{R}^{n+2})}&\lesssim (1+|t|)^{-\theta}\|f\|_{L^1_{x,y}(\mathbb{R}^{n+2})},
\end{align*}
which implies that
\begin{align*}
\|U_k(t)f\|_{L^\infty_{x,y}(\mathbb{R}^{n+2})}&\lesssim 2^{\frac{n+3+\eta}{2}k}(2^{-k}+|t|)^{-\frac{n+1+\eta}{2}}\|f\|_{L^1_{x,y}(\mathbb{R}^{n+2})},\\
\|U^{low}(t)f\|_{L^\infty_{x,y}(\mathbb{R}^{n+2})}&\lesssim (1+|t|)^{-\frac{n+1+\eta}{2}}\|f\|_{L^1_{x,y}(\mathbb{R}^{n+2})},
\end{align*}
Then we have for any $k\geq1$ and $0\leq \eta\leq 1$,
\begin{align*}
\|U_k(t)U_k(s)^*\|_{L^1_{x,y}(\mathbb{R}^{n+2})\rightarrow L^\infty_{x,y}(\mathbb{R}^{n+2})}&\lesssim 2^{\frac{n+3+\eta}{2}k}(2^{-k}+|t-s|)^{-\frac{n+1+\eta}{2}},\\
\|U^{low}(t)U^{low}(s)^*\|_{L^1_{x,y}(\mathbb{R}^{n+2})\rightarrow L^\infty_{x,y}(\mathbb{R}^{n+2})}&\lesssim (1+|t-s|)^{-\frac{n+1+\eta}{2}},
\end{align*}
where $U_k(t)$ satisfies the estimates \eqref{1 to infty} in Proposition \ref{Keel-Tao argument} for $\alpha=\frac{n+3+\eta}{2}$, $\sigma=\frac{n+1+\eta}{2}$ and $h=2^{-k}$. 
On the other hand, by Plancherel's formula \eqref{Plancherel}, we get the energy estimate
\begin{equation*}
\|U^{low}(t)\|_{L^2_{x,y}(\mathbb{R}^{n+2})\rightarrow L^2_{x,y}(\mathbb{R}^{n+2})}\lesssim 1,\quad \|U_k(t)\|_{L^2_{x,y}(\mathbb{R}^{n+2})\rightarrow L^2_{x,y}(\mathbb{R}^{n+2})}\lesssim 1,\quad \forall k\geq1.
\end{equation*}
Using Keel-Tao's argument in Proposition \ref{Keel-Tao argument} and Theorem 1.2 in Keel-Tao \cite{KT}, it gives that
\begin{align*}
\|U_k(t)f\|_{L_t^q(\mathbb{R},L^r_{x,y}(\mathbb{R}^{2+n}))}&\lesssim 2^{k[(n+2+\eta)(\frac{1}{2}-\frac{1}{r})-\frac{1}{q}]}\|f\|_{L^2(\mathbb{R}^{2+n})}, \quad\forall k\geq1,\\
\|U^{low}(t)f\|_{L_t^q(\mathbb{R},L^r_{x,y}(\mathbb{R}^{2+n}))}&\lesssim \|f\|_{L^2(\mathbb{R}^{2+n})},
\end{align*}
for any $0\leq \eta\leq1$ and $q\in[2,\infty]$, $r\in [2,\infty)$ such that $\frac{2}{q}+\frac{n+1+\eta}{r}\leq \frac{n+1+\eta}{2}$. Notice that
\begin{equation*}
e^{it\sqrt{1+\mathcal{L}_a}}f=U^{low}(t)f+\sum_{k=1}^\infty U_k(t)f,
\end{equation*}
and $U_k(t)\psi(2^{-j}\sqrt{\mathcal{L}_a})f$ vanishes if $|j-k|\geq3$. Hence, 
\begin{align*}
\|U^{low}(t)f\|_{L_t^q(\mathbb{R},L^r_{x,y}(\mathbb{R}^{n+2}))}&=\|U^{low}(t)(\varphi(\sqrt{\mathcal{L}_a})+\sum_{j=1}^2\psi(2^{-j}\sqrt{\mathcal{L}_a}))f\|_{L_t^q(\mathbb{R},L^r_{x,y}(\mathbb{R}^{n+2}))}\\
&\lesssim \|(\varphi(\sqrt{\mathcal{L}_a})+\sum_{j=1}^2\psi(2^{-j}\sqrt{\mathcal{L}_a}))f\|_{L^2_{x,y}(\mathbb{R}^{n+2})}\\
&\leq \|\varphi(\sqrt{\mathcal{L}_a})f\|_{L^2_{x,y}(\mathbb{R}^{n+2})}+\sum_{j=1}^2\|\psi(2^{-j}\sqrt{\mathcal{L}_a})f\|_{L^2_{x,y}(\mathbb{R}^{n+2})},
\end{align*}
and
\begin{align*}
&\quad \|\sum_{k=1}^\infty U_k(t)f\|_{L_t^q(\mathbb{R},L^r_{x,y}(\mathbb{R}^{n+2}))}\\
&=\|\sum_{j\in\mathbb{Z}}\sum_{k=1}^\infty U_k(t)\psi(2^{-j}\sqrt{\mathcal{L}_a})f\|_{L_t^q(\mathbb{R},L^r_{x,y}(\mathbb{R}^{n+2}))}\\
&=\Big\|\sum_{j=-1}^3U_1(t)\psi(2^{-j}\sqrt{\mathcal{L}_a})f+\sum_{j=0}^4U_2(t)\psi(2^{-j}\sqrt{\mathcal{L}_a})f\\
&\qquad\qquad\qquad\qquad\qquad\qquad\qquad+\sum_{j\in\mathbb{Z}}^\infty\sum_{k=3}^\infty U_k(t)\psi(2^{-j}\sqrt{\mathcal{L}_a})f\Big\|_{L_t^q(\mathbb{R},L^r_{x,y}(\mathbb{R}^{n+2}))}.\\
&\leq \sum_{j=-1}^3\|U_1(t)\psi(2^{-j}\sqrt{\mathcal{L}_a})f\|_{L_t^q(\mathbb{R},L^r_{x,y}(\mathbb{R}^{n+2}))}+\sum_{j=0}^4\|U_2(t)\psi(2^{-j}\sqrt{\mathcal{L}_a})f\|_{L_t^q(\mathbb{R},L^r_{x,y}(\mathbb{R}^{n+2}))}\\
&\qquad\qquad\qquad\qquad\qquad\qquad\qquad+\|\sum_{j\in\mathbb{Z}}^\infty\sum_{k=3}^\infty U_k(t)\psi(2^{-j}\sqrt{\mathcal{L}_a})f\|_{L_t^q(\mathbb{R},L^r_{x,y}(\mathbb{R}^{n+2}))}.\\
&\lesssim \sum_{j=-1}^4\|\psi(2^{-j}\sqrt{\mathcal{L}_a})f\|_{L^2_{x,y}(\mathbb{R}^{n+2})}+\|\sum_{j\in\mathbb{Z}}^\infty\sum_{k=3}^\infty U_k(t)\psi(2^{-j}\sqrt{\mathcal{L}_a})f\|_{L_t^q(\mathbb{R},L^r_{x,y}(\mathbb{R}^{n+2}))}.
\end{align*}
Using the Littlewood-Paley square inequality in Proposition \ref{LP inequality} and the Minkowski inequality, we have
\begin{align*}
&\quad\|\sum_{j\in\mathbb{Z}}^\infty\sum_{k=3}^\infty U_k(t)\psi(2^{-j}\sqrt{\mathcal{L}_a})f\|_{L_t^q(\mathbb{R},L^r_{x,y}(\mathbb{R}^{n+2}))}\\
&\lesssim \|\Big(\sum_{j\in\mathbb{Z}}^\infty |\sum_{k=3}^\infty U_k(t)\psi(2^{-j}\sqrt{\mathcal{L}_a})f|^2\Big)^\frac{1}{2}\|_{L_t^q(\mathbb{R},L^r_{x,y}(\mathbb{R}^{n+2}))}\\
&\lesssim \|\Big(\sum_{j\in\mathbb{Z}}^\infty \sum_{\substack{k\geq 3\\ |k-j|\leq 2}} |U_k(t)\psi(2^{-j}\sqrt{\mathcal{L}_a})f|^2\Big)^\frac{1}{2}\|_{L_t^q(\mathbb{R},L^r_{x,y}(\mathbb{R}^{n+2}))}\\
&\leq \Big(\sum_{j\in\mathbb{Z}}^\infty \sum_{\substack{k\geq 3\\ |k-j|\leq 2}} \|U_k(t)\psi(2^{-j}\sqrt{\mathcal{L}_a})f\|_{L_t^q(\mathbb{R},L^r_{x,y}(\mathbb{R}^{n+2}))}^2\Big)^\frac{1}{2}\\
&\lesssim \|\Big(\sum_{j\in\mathbb{Z}}^\infty \sum_{\substack{k\geq 3\\ |k-j|\leq 2}} 2^{2k[(n+2+\eta)(\frac{1}{2}-\frac{1}{r})-\frac{1}{q}]}\|\psi(2^{-j}\sqrt{\mathcal{L}_a})f\|^2_{L^2(\mathbb{R}^{2+n})}\Big)^\frac{1}{2}\|_{L_t^q(\mathbb{R},L^r_{x,y}(\mathbb{R}^{n+2}))}\\
&\lesssim \Big(\sum_{j\geq1}^\infty 2^{2j[(n+2+\eta)(\frac{1}{2}-\frac{1}{r})-\frac{1}{q}]}\|\psi(2^{-j}\sqrt{\mathcal{L}_a})f\|^2_{L^2(\mathbb{R}^{2+n})}\Big)^\frac{1}{2}.
\end{align*}
Therefore, we get
\begin{align*}
&\quad \|e^{it\sqrt{1+\mathcal{L}_a}}f\|_{L_t^q(\mathbb{R},L^r_{x,y}(\mathbb{R}^{n+2}))}\\
&\lesssim \|U^{low}(t)f\|_{L_t^q(\mathbb{R},L^r_{x,y}(\mathbb{R}^{n+2}))}+\|\sum_{k=1}^\infty U_k(t)f\|_{L_t^q(\mathbb{R},L^r_{x,y}(\mathbb{R}^{n+2}))}\\
&\lesssim\|\varphi(\sqrt{\mathcal{L}_a})f\|_{L^2_{x,y}(\mathbb{R}^{n+2})}+\sum_{j=-1}^4\|\psi(2^{-j}\sqrt{\mathcal{L}_a})f\|_{L^2_{x,y}(\mathbb{R}^{n+2})}\\
&\qquad\qquad\qquad\qquad+\Big(\sum_{j\geq1}^\infty 2^{2j[(n+2+\eta)(\frac{1}{2}-\frac{1}{r})-\frac{1}{q}]}\|\psi(2^{-j}\sqrt{\mathcal{L}_a})f\|^2_{L^2(\mathbb{R}^{2+n})}\Big)^\frac{1}{2}\\
&\lesssim\|\varphi(\sqrt{\mathcal{L}_a})f\|_{L^2_{x,y}(\mathbb{R}^{n+2})}+\Big(\sum_{j\geq1}^\infty 2^{2j[(n+2+\eta)(\frac{1}{2}-\frac{1}{r})-\frac{1}{q}]}\|\psi(2^{-j}\sqrt{\mathcal{L}_a})f\|^2_{L^2(\mathbb{R}^{2+n})}\Big)^\frac{1}{2}\\
&=\|f\|_{H_a^{(n+2+\eta)(\frac{1}{2}-\frac{1}{r})-\frac{1}{q}}},
\end{align*}
which completes the proof.
\end{proof}
By Theorem \ref{Klein-Gordon estimate}, we immediately deduce the Strichartz inequality for the solution of the Klein-Gordon equation \eqref{K-GEqu}. 
\begin{corollary}
Under the same hypotheses as in Theorem \ref{Klein-Gordon estimate}, the solution $u$ of the Klein-Gordon equation \eqref{K-GEqu} satisfies the following estimate
	\begin{equation*}
	\|u\|_{L_t^q(\mathbb{R},L^r_{x,y}(\mathbb{R}^{2+n}))}\lesssim \|f\|_{H^s_a(\mathbb{R}^{2+n})}+\|g\|_{H^{s-1}_a(\mathbb{R}^{2+n})}.
	\end{equation*}
\end{corollary}

 \subsection{Beam equation}
Finally, we consider the beam equation
\begin{equation}\label{BeamEqu}
	\begin{cases}
	\partial_t^2u+\mathcal{L}^2_a u+u=0,\;t\in\mathbb{R},\; (x,y)\in\mathbb{R}^{2+n},\\
	u(0,x,y)=f(x,y),\; \partial_tu(0,x,y)=g(x,y).
	\end{cases}
	\end{equation}
By Duhamel's principle, the solution is formally given by
\begin{equation*}
	u(t,x,y)=\frac{e^{it\sqrt{1+\mathcal{L}^2_a}}+e^{-it\sqrt{1+\mathcal{L}^2_a}}}{2}f(x,y)+\frac{e^{it\sqrt{1+\mathcal{L}^2_a}}-e^{-it\sqrt{1+\mathcal{L}^2_a}}}{2i\sqrt{1+\mathcal{L}^2_a}}g(x,y).
	\end{equation*}
So we shall discuss the operator $e^{it\sqrt{1+\mathcal{L}^2_a}}$, which corresponds to the case when $\phi(\lambda)=\sqrt{1+\lambda^4}$.
By a simple calculation, we see that
	\begin{equation*}
	\phi'(\lambda)=2\lambda^3(1+\lambda^4)^{-\frac{1}{2}}, \quad
	\phi''(\lambda)=(6\lambda^2+2\lambda^6)(1+\lambda^4)^{-\frac{3}{2}}.
	\end{equation*}
This shows that $\phi$ satisfies (C1)-(C4) with $m_1=\alpha_1=2, m_2=\alpha_2=4$.  Using  Theorem \ref{ResultTime}, we obtain the following Strichartz estimate for $e^{it\sqrt{1+\mathcal{L}^2_a}}$.
\begin{theorem}\label{Beam estimate}
Let $s\geq 0$, $q\in [2,\infty]$ and  $r\in [2,\infty)$ satisfy
\begin{equation*}
\frac{2}{q}+\frac{n+2}{r}\leq \frac{n+2}{2},\quad s=(n+2)\left(\frac{1}{2}-\frac{1}{r}\right)-\frac{2}{q}.
\end{equation*}
Then we have
\begin{equation*}
\|e^{it\sqrt{1+\mathcal{L}^2_a}}f\|_{L_t^q(\mathbb{R},L^r_{x,y}(\mathbb{R}^{n+2}))}\lesssim \|f\|_{H^s_a(\mathbb{R}^{n+2})}.
\end{equation*}
\end{theorem}
\begin{proof} Denote $U^{low}(t)=e^{it\sqrt{1+\mathcal{L}^2_a}}\varphi(\sqrt{\mathcal{L}_a})$ and $U_k(t)=e^{it\sqrt{1+\mathcal{L}^2_a}}\psi(2^{-k}\sqrt{\mathcal{L}_a})$, for any $k\in\mathbb{Z}$.
It follows from Theorem \ref{ResultTime} that for any $0\leq \theta\leq \frac{n+2}{2}$ and $k\geq1$,
\begin{align*}
\|U_k(t)f\|_{L^\infty_{x,y}(\mathbb{R}^{n+2})}&\lesssim |t|^{-\theta}2^{k(n+2-2\theta)}\|f\|_{L^1_{x,y}(\mathbb{R}^{n+2})},\\
\|U^{low}(t)f\|_{L^\infty_{x,y}(\mathbb{R}^{n+2})}&\lesssim (1+|t|)^{-\frac{\theta}{2}}\|f\|_{L^1_{x,y}(\mathbb{R}^{n+2})},
\end{align*}
which implies that
\begin{align*}
\|U_k(t)f\|_{L^\infty_{x,y}(\mathbb{R}^{n+2})}&\lesssim (2^{-2k}+|t|)^{-\frac{n+2}{2}}\|f\|_{L^1_{x,y}(\mathbb{R}^{n+2})},\\
\|U^{low}(t)f\|_{L^\infty_{x,y}(\mathbb{R}^{n+2})}&\lesssim (1+|t|)^{-\frac{n+2}{4}}\|f\|_{L^1_{x,y}(\mathbb{R}^{n+2})},
\end{align*}
Then we have
\begin{align*}
\|U_k(t)U_k(s)^*\|_{L^1_{x,y}(\mathbb{R}^{n+2})\rightarrow L^\infty_{x,y}(\mathbb{R}^{n+2})}&\lesssim (2^{-2k}+|t-s|)^{-\frac{n+2}{2}},\quad \forall k\geq1,\\
\|U^{low}(t)U^{low}(s)^*\|_{L^1_{x,y}(\mathbb{R}^{n+2})\rightarrow L^\infty_{x,y}(\mathbb{R}^{n+2})}&\lesssim (1+|t-s|)^{-\frac{n+2}{4}},
\end{align*}
where $U_k(t),k\geq1$ satisfies the estimates \eqref{1 to infty} in Proposition \ref{Keel-Tao argument} for $\alpha=0$, $\sigma=\frac{n+2}{2}$ and $h=2^{-2k}$. 
On the other hand, by Plancherel's formula {\color{red} \eqref{Plancherel}}, we get the energy estimate
\begin{equation*}
\|U^{low}(t)\|_{L^2_{x,y}(\mathbb{R}^{n+2})\rightarrow L^2_{x,y}(\mathbb{R}^{n+2})}\lesssim 1,\quad \|U_k(t)\|_{L^2_{x,y}(\mathbb{R}^{n+2})\rightarrow L^2_{x,y}(\mathbb{R}^{n+2})}\lesssim 1,\quad \forall k\geq1.
\end{equation*}
Using Keel-Tao's argument in Proposition \ref{Keel-Tao argument} and Theorem 1.2 in Keel-Tao \cite{KT}, it gives that
\begin{align*}
\|U_k(t)f\|_{L_t^q(\mathbb{R},L^r_{x,y}(\mathbb{R}^{2+n}))}&\lesssim 2^{k[(n+2)(\frac{1}{2}-\frac{1}{r})-\frac{2}{q}]}\|f\|_{L^2(\mathbb{R}^{2+n})}, \quad\forall k\geq1,\\
\|U^{low}(t)f\|_{L_t^q(\mathbb{R},L^r_{x,y}(\mathbb{R}^{2+n}))}&\lesssim \|f\|_{L^2(\mathbb{R}^{2+n})},
\end{align*}
for $q\in[2,\infty]$, $r\in [2,\infty)$ such that $\frac{2}{q}+\frac{n+2}{r}\leq \frac{n+2}{2}$. Notice that
\begin{equation*}
e^{it\sqrt{1+\mathcal{L}^2_a}}f=U^{low}(t)f+\sum_{k=1}^\infty U_k(t)f,
\end{equation*}
and $U_k(t)\psi(2^{-j}\sqrt{\mathcal{L}_a})f$ vanishes if $|j-k|\geq3$. Hence, 
\begin{align*}
\|U^{low}(t)f\|_{L_t^q(\mathbb{R},L^r_{x,y}(\mathbb{R}^{n+2}))}&=\|U^{low}(t)(\varphi(\sqrt{\mathcal{L}_a})+\sum_{j=1}^2\psi(2^{-j}\sqrt{\mathcal{L}_a}))f\|_{L_t^q(\mathbb{R},L^r_{x,y}(\mathbb{R}^{n+2}))}\\
&\lesssim \|(\varphi(\sqrt{\mathcal{L}_a})+\sum_{j=1}^2\psi(2^{-j}\sqrt{\mathcal{L}_a}))f\|_{L^2_{x,y}(\mathbb{R}^{n+2})}\\
&\leq \|\varphi(\sqrt{\mathcal{L}_a})f\|_{L^2_{x,y}(\mathbb{R}^{n+2})}+\sum_{j=1}^2\|\psi(2^{-j}\sqrt{\mathcal{L}_a})f\|_{L^2_{x,y}(\mathbb{R}^{n+2})},
\end{align*}
and
\begin{align*}
&\quad \|\sum_{k=1}^\infty U_k(t)f\|_{L_t^q(\mathbb{R},L^r_{x,y}(\mathbb{R}^{n+2}))}\\
&=\|\sum_{j\in\mathbb{Z}}\sum_{k=1}^\infty U_k(t)\psi(2^{-j}\sqrt{\mathcal{L}_a})f\|_{L_t^q(\mathbb{R},L^r_{x,y}(\mathbb{R}^{n+2}))}\\
&=\Big\|\sum_{j=-1}^3U_1(t)\psi(2^{-j}\sqrt{\mathcal{L}_a})f+\sum_{j=0}^4U_2(t)\psi(2^{-j}\sqrt{\mathcal{L}_a})f\\
&\qquad\qquad\qquad\qquad\qquad\qquad\qquad+\sum_{j\in\mathbb{Z}}^\infty\sum_{k=3}^\infty U_k(t)\psi(2^{-j}\sqrt{\mathcal{L}_a})f\Big\|_{L_t^q(\mathbb{R},L^r_{x,y}(\mathbb{R}^{n+2}))}.\\
&\leq \sum_{j=-1}^3\|U_1(t)\psi(2^{-j}\sqrt{\mathcal{L}_a})f\|_{L_t^q(\mathbb{R},L^r_{x,y}(\mathbb{R}^{n+2}))}+\sum_{j=0}^4\|U_2(t)\psi(2^{-j}\sqrt{\mathcal{L}_a})f\|_{L_t^q(\mathbb{R},L^r_{x,y}(\mathbb{R}^{n+2}))}\\
&\qquad\qquad\qquad\qquad\qquad\qquad\qquad+\|\sum_{j\in\mathbb{Z}}^\infty\sum_{k=3}^\infty U_k(t)\psi(2^{-j}\sqrt{\mathcal{L}_a})f\|_{L_t^q(\mathbb{R},L^r_{x,y}(\mathbb{R}^{n+2}))}.\\
&\lesssim \sum_{j=-1}^4\|\psi(2^{-j}\sqrt{\mathcal{L}_a})f\|_{L^2_{x,y}(\mathbb{R}^{n+2})}+\|\sum_{j\in\mathbb{Z}}^\infty\sum_{k=3}^\infty U_k(t)\psi(2^{-j}\sqrt{\mathcal{L}_a})f\|_{L_t^q(\mathbb{R},L^r_{x,y}(\mathbb{R}^{n+2}))}.
\end{align*}
Using the Littlewood-Paley square inequality in Proposition \ref{LP inequality} and the Minkowski inequality, we have
\begin{align*}
&\quad\|\sum_{j\in\mathbb{Z}}^\infty\sum_{k=3}^\infty U_k(t)\psi(2^{-j}\sqrt{\mathcal{L}_a})f\|_{L_t^q(\mathbb{R},L^r_{x,y}(\mathbb{R}^{n+2}))}\\
&\lesssim \|\Big(\sum_{j\in\mathbb{Z}}^\infty |\sum_{k=3}^\infty U_k(t)\psi(2^{-j}\sqrt{\mathcal{L}_a})f|^2\Big)^\frac{1}{2}\|_{L_t^q(\mathbb{R},L^r_{x,y}(\mathbb{R}^{n+2}))}\\
&\lesssim \|\Big(\sum_{j\in\mathbb{Z}}^\infty \sum_{\substack{k\geq 3\\ |k-j|\leq 2}} |U_k(t)\psi(2^{-j}\sqrt{\mathcal{L}_a})f|^2\Big)^\frac{1}{2}\|_{L_t^q(\mathbb{R},L^r_{x,y}(\mathbb{R}^{n+2}))}\\
&\leq \Big(\sum_{j\in\mathbb{Z}}^\infty \sum_{\substack{k\geq 3\\ |k-j|\leq 2}} \|U_k(t)\psi(2^{-j}\sqrt{\mathcal{L}_a})f\|_{L_t^q(\mathbb{R},L^r_{x,y}(\mathbb{R}^{n+2}))}^2\Big)^\frac{1}{2}\\
&\lesssim \|\Big(\sum_{j\in\mathbb{Z}}^\infty \sum_{\substack{k\geq 3\\ |k-j|\leq 2}} 2^{2k[(n+2)(\frac{1}{2}-\frac{1}{r})-\frac{2}{q}]}\|\psi(2^{-j}\sqrt{\mathcal{L}_a})f\|^2_{L^2(\mathbb{R}^{2+n})}\Big)^\frac{1}{2}\|_{L_t^q(\mathbb{R},L^r_{x,y}(\mathbb{R}^{n+2}))}\\
&\lesssim \Big(\sum_{j\geq1}^\infty 2^{2j[(n+2)(\frac{1}{2}-\frac{1}{r})-\frac{2}{q}]}\|\psi(2^{-j}\sqrt{\mathcal{L}_a})f\|^2_{L^2(\mathbb{R}^{2+n})}\Big)^\frac{1}{2},
\end{align*}
Therefore, we get
\begin{align*}
&\quad \|e^{it\sqrt{1+\mathcal{L}^2_a}}f\|_{L_t^q(\mathbb{R},L^r_{x,y}(\mathbb{R}^{n+2}))}\\
&\lesssim \|U^{low}(t)f\|_{L_t^q(\mathbb{R},L^r_{x,y}(\mathbb{R}^{n+2}))}+\|\sum_{k=1}^\infty U_k(t)f\|_{L_t^q(\mathbb{R},L^r_{x,y}(\mathbb{R}^{n+2}))}\\
&\lesssim\|\varphi(\sqrt{\mathcal{L}_a})f\|_{L^2_{x,y}(\mathbb{R}^{n+2})}+\sum_{j=-1}^4\|\psi(2^{-j}\sqrt{\mathcal{L}_a})f\|_{L^2_{x,y}(\mathbb{R}^{n+2})}\\
&\qquad\qquad\qquad\qquad+\Big(\sum_{j\geq1}^\infty 2^{2j[(n+2)(\frac{1}{2}-\frac{1}{r})-\frac{2}{q}]}\|\psi(2^{-j}\sqrt{\mathcal{L}_a})f\|^2_{L^2(\mathbb{R}^{2+n})}\Big)^\frac{1}{2}\\
&\lesssim\|\varphi(\sqrt{\mathcal{L}_a})f\|_{L^2_{x,y}(\mathbb{R}^{n+2})}+\Big(\sum_{j\geq1}^\infty 2^{2j[(n+2)(\frac{1}{2}-\frac{1}{r})-\frac{2}{q}]}\|\psi(2^{-j}\sqrt{\mathcal{L}_a})f\|^2_{L^2(\mathbb{R}^{2+n})}\Big)^\frac{1}{2}\\
&=\|f\|_{H_a^{(n+2)(\frac{1}{2}-\frac{1}{r})-\frac{2}{q}}},
\end{align*}
which completes the proof.
\end{proof}
By Theorem \ref{Beam estimate}, we immediately deduce the Strichartz inequality for the solution of the beam equation \eqref{BeamEqu}. 
\begin{corollary}
Under the same hypotheses as in Theorem \ref{Beam estimate}, the solution $u$ of the beam equation \eqref{BeamEqu} satisfies the following estimate
	\begin{equation*}
	\|u\|_{L_t^q(\mathbb{R},L^r_{x,y}(\mathbb{R}^{2+n}))}\lesssim \|f\|_{H^s_a(\mathbb{R}^{2+n})}+\|g\|_{H^{s-1}_a(\mathbb{R}^{2+n})}.
	\end{equation*}
\end{corollary}

\section*{Statements and Declarations} The authors confirm that the data supporting the findings of this study are available within the article and its supplementary materials.

  \textbf{Competing Interests:} No potential competing of interest was reported by the author. 


\begin{thebibliography}{99}
		
		
		\normalsize
		\baselineskip=17pt

  
\bibitem{AP2009}\label{AP2009} J. -P. Anker and V. Pierfelice, \emph{Nonlinear Schr\"odinger equation on real hyperbolic spaces}, Ann. Inst. H.
Poincar\'e (C) Non Linear Anal. \textbf{26}(5), 1853-1869 (2009).

\bibitem{AP2014}\label{AP2014} J. -P. Anker and V. Pierfelice, \emph{Wave and Klein-Gordon equations on hyperbolic spaces}, Anal. PDE \textbf{7}, 953-995 (2014).

\bibitem{APV2012}\label{APV2012} J. -P. Anker, V. Pierfelice, and M. Vallarino, \emph{The wave equation on hyperbolic spaces}, J. Differ. Equ. \textbf{252}(10), 5613-5661 (2012).

\bibitem{BBG2021}\label{BBG2001} H. Bahouri, D. Barilari, and I. Gallagher, \emph{Strichartz estimates and Fourier restriction theorems on the Heisenberg group}, J. Fourier Anal. Appl. \textbf{27}(2) (2021).

\bibitem{Bahouri-app} H. Bahouri, J.-Y. Chemin,  and R. Danchin, \emph{Fourier analysis and nonlinear partial differential equations},  Grundlehrender Math. Wissenschaften 343, Springer, Heidelberg, 2011.

\bibitem{BGX2000}\label{BGX2000} H. Bahouri, P. G\'{e}rard, and C. J. Xu, \emph{Espaces de Besov et estimations de Strichartz g\'{e}n\'{e}ralis\'{e}es sur le groupe de
Heisenberg}, J. Anal. Math. \textbf{82}, 93-118 (2000).

\bibitem{BKG}\label{BKG} H. Bahouri, C. F. Kammerer, and I. Gallagher, \emph{Dispersive estimates for the Schr\"{o}dinger operator on step 2 stratified Lie groups}, Anal. PDE \textbf{9}, 545-574 (2016).

\bibitem{Ben said}\label{Ben said} 	 S. Ben Sa\"id, \emph{Strichartz estimates for Schr\"odinger-Laguerre operators}, {Semigroup Forum} \textbf{90}, 251-269 (2015). 
		
\bibitem{Ratna3}  S. Ben Sa\"id, A. K. Nandakumaran, and  P. K. Ratnakumar, \emph{Schr\"odinger propagator and the Dunkl Laplacian},    hal-00578446v1, (2011).


\bibitem{B1993}\label{B1993} J. Bourgain, \emph{Fourier transform restriction phenomena for certain lattice subsets and application to
nonlinear evolution equations \uppercase \expandafter {\romannumeral 1}}, Geom. Funct. Anal. \textbf{3}, 107-156 (1993).

\bibitem{BDDM2019}\label{BDDM2019} T. A. Bui, P. D\text{'}Ancona, X. T. Duong, and  D. M\"{u}ller, \emph{On the flows associated to self-adjoint operators on metric measure spaces}, Math. Ann. \textbf{375}, 1393-1426 (2019).

\bibitem{BGT2004}\label{BGT2004} N. Burq, P. G\'{e}rard, and N. Tzvetkov, \emph{Strichartz inequalities and the nonlinear Schr\"odinger equation on
compact manifolds}, Amer. J. Math. \textbf{126}, 569-605 (2004).

\bibitem{burq2003} N. Burq, F. Planchon, J.G. Stalker, A.S. Tahvildar-Zadeh, 
\emph{Strichartz estimates for the wave and Schrödinger equations with the inverse-square potential}, J. Funct. Anal. \textbf{203}, 519--549 (2003).


\bibitem{DPR2010}\label{DPR2010} P. D\text{'}Ancona, V. Pierfelice, and F. Ricci, \emph{On the wave equation associated to the Hermite and the twisted Laplacian}, J. Fourier Anal. Appl. \textbf{16}(2), 294-310 (2010).

\bibitem{FMSW}G. Feng, S. S. Mondal, M. Song, H. Wu, \emph{Orthonormal Strichartz inequalities and their applications on abstract measure spaces}, Math. Z. \textbf{312}(55), Paper No. 55, 42 pp (2026).

\bibitem{Feng-Song}\label{Feng-Song} G. Feng, M. Song, \emph{Restriction theorems and Strichartz inequalities for the Laguerre operator involving orthonormal functions}, J. Geom. Anal. \textbf{34}, 287 (2024).

\bibitem{FSW}\label{FSW} G. Feng, M. Song, and H. Wu, \emph{Decay estimates for a class of semigroups related to self-adjoint operators on metric measure spaces}, J. Geom. Anal. \textbf{35}(7), Paper No. 192 (2025).

\bibitem{FMV1}\label{FMV1}G. Furioli, C. Melzi, and A. Veneruso, \emph{Strichartz inequalities for the wave equation with the full Laplacian on the Heisenberg group},  Canad. J. Math. \textbf{59}(6), 1301-1322 (2007).

\bibitem{FV}\label{FV} G. Furioli and A. Veneruso, \emph{Strichartz inequalities for the Schr\"{o}dinger equation with the full Laplacian on the Heisenberg group}, Studia Math. \textbf{160}, 157-178 (2004).

\bibitem{GV}J. Ginibre and  G. Velo, \emph{Generalized Strichartz inequalities for the wave equation}, J. Funct. Anal. \textbf{133}, 50-68 (1995).

\bibitem{GLNY}  Z. Guo, J. Li, K. Nakanishi and L. Yan, \emph{On the boundary Strichartz estimates for wave and Schr\"odinger equations}, J. Differential Equations \textbf{265} (11), 5656-5675 (2018).

\bibitem{GPW2008} Z. Guo, L. Peng, and B. Wang, \emph{Decay estimates for a class of wave equations}, J. Funct. Anal. \textbf{254}(6), 1642-1660 (2008).

\bibitem{H2005}\label{H2005} M. D. Hierro, \emph{Dispersive and Strichartz estimates on H-type groups}, Studia Math. \textbf{169}, 1-20 (2005).

\bibitem{ILP2014} O. Ivanovici, G. Lebeau, and F. Planchon, \emph{Dispersion for the wave equation inside strictly convex domains \uppercase \expandafter {\romannumeral 1}: the Friedlander model case}, Ann. Math. \textbf{180}, 323-380 (2014).

\bibitem{KT}\label{KT} M. Keel and  T. Tao, \emph{Endpoints Strichartz estimates}, Amer. J. Math. \textbf{120}, 955-980 (1998). 

\bibitem{Koch-Tataru-Zworski} H. Koch, D. Tataru, M. Zworski, \emph{Semiclassical $L^p$ estimates}, Ann. Henri Poincar\'e \textbf{8}, 885-916 (2007).

\bibitem{LS2014} H. Liu and M. Song, \emph{Strichartz inequalities for the wave equation with the full Laplacian on H-type groups}, Abstr. Appl. Anal., Art. ID 219375, 10pp (2014).

\bibitem{Mejjaoli2008-1} H. Mejjaoli, 
\emph{Strichartz estimates for the Dunkl wave equation and application},
J. Math. Anal. Appl. \textbf{346}(1), 41–54 (2008).

\bibitem{Mejjaoli2009} H. Mejjaoli, \emph{Dispersion phenomena in Dunkl-Schr\"{o}dinger equation and applications},
Serdica Math. J. \textbf{35}(1), 25–60 (2009).

\bibitem{Mejjaoli2013}  H. Majjaoli,  \emph{Dunkl-Schr\"odinger semigroups and applications},  Appl. Anal. \textbf{92}(8), 1597-1626 (2013).

\bibitem{MS}S. S. Mondal and  M. Song, \emph{Orthonormal Strichartz inequalities for the $(k,a)$-generalized Laguerre operator and Dunkl operator},  Israel J. Math. \textbf{269}(2), 697-729 (2025).  

\bibitem{NR2005}\label{NR2005} A. K. Nandakumaran, and P. K. Ratnakumar, \emph{Schr\"{o}dinger equation and the oscillatory semigroup for the Hermite operator}, J. Funct. Anal. \textbf{224}(2), 371-385 (2005).

\bibitem{R2008}\label{R2008} P. K. Ratnakumar, \emph{On Schr\"{o}dinger Propagator for the Special Hermite Operator}, J. Fourier Anal. Appl. \textbf{14}(2), 286-300 (2008).

\bibitem{S2013}V. K. Sohani, \emph{Strichartz estimates for the Schr\"{o}dinger propagator for the Laguerre operator}, Proc. Math. Sci. \textbf{123}, 525-537 (2013).

\bibitem{Song2026} M. Song, \emph{Strichartz inequalities for the Schr\"odinger equation with the full Laplacian on H-type groups}, J. Geom. Phys. \textbf{228}, 105925 (2026).

\bibitem{Song2016} M. Song, \emph{Decay estimates for fractional wave equations on H-type groups}, J. Inequal. Appl. \textbf{2016}(246), 12pp (2016).

\bibitem{Song-Tan}\label{Song-Tan} M. Song, J. Tan, \emph{Decay estimates and Strichartz inequalities for a class of dispersive equations on H-type groups}, Discrete Contin. Dyn. Syst. \textbf{50}, 103-131(2026).

\bibitem{SY2023}M. Song and J. Yang, \emph{Decay estimates for a class of wave equations on the Heisenberg group}, Ann. Mat. Pur. Appl. \textbf{202}, 2665–2685(2023).


\bibitem{SZ}N. Song and  J. Zhao, \emph{Strichartz estimates on the quaternion Heisenberg group}, Bull. Sci. Math. \textbf{138}(2), 293-315 (2014).

\bibitem{Stein1984} E. M. Stein,  \textit{Oscillatory integrals in Fourier analysis}.  In: Beijing Lectures in Harmonic Analysis (Beijing, 1984),  Annals of Mathematics Studies, vol. 112,  Princeton University Press, Princeton, 1986.

\bibitem{S1993}\label{S1993} E. M. Stein, \emph{Harmonic analysis: real-variable methods, orthogonality and oscillatory integrals}, Princeton Univ. Press, 1993.

\bibitem{Str}\label{Str} R. S. Strichartz, \emph{Restrictions of Fourier transforms to quadratic surfaces and
	decay of solutions of wave equations}, Duke Math. J. \textbf{44}, 705-714 (1977).

\bibitem{taylor1996} M. Taylor, \emph{Partial Differential Equations, Vol. II}. Springer (1996).

\bibitem{T1} P. A. Tomas,   \emph{A restriction theorem for the Fourier transform},  {Bull. Amer. Math. Soc.} \textbf{81}, 477-478 (1975).

\bibitem{T2} P. A. Tomas,   \emph{Restriction theorems for the Fourier transform},   In: Proc. Symp. Pure Math.,  XXXV (1979).

\bibitem{WXZZ} J. Wang, C. Xu, F. Zhang, J. Zhang, \emph{$L^p$-estimates for the wave equation with partial inverse-square potentials}, arXiv:2603.27111 (2026)

\bibitem{Zhang-Zhang}\label{Zhang-Zhang} F. Zhang, J. Zhang, \emph{Strichartz estimates for dispersive equations with partial inverse-square potentials}, J. Geom. Anal. \textbf{35}(3), Paper No. 71, 27pp (2025).

\bibitem{Zhang2019} J. Zhang, \emph{Strichartz estimates and nonlinear wave equation on nontrapping asymptotically conic manifolds}, Adv. Math. \textbf{271}, 91–111 (2015).

\bibitem{Zworski} M. Zworski, \emph{Semiclassical Analysis}, Grad. Stud. Math., vol. 138, AMS, 2012.

	\end{thebibliography}
\end{document}